%% file: paper.tex
\documentclass[a4paper]{amsart}

\usepackage{amssymb}
\usepackage{amsthm}
\usepackage{mathtools}
\usepackage{mathrsfs} 
\usepackage[normalem]{ulem} 
\usepackage{float} 
\usepackage{enumerate} 
\usepackage{wasysym} 
\usepackage{enumitem}
\usepackage{multirow} 
\usepackage{hyperref} 
\usepackage{xparse} 
\usepackage{xargs} 
\usepackage{todonotes} 
\usepackage[margin=24mm, a4paper]{geometry}
\usepackage[foot]{amsaddr}

\newcommand\definition[1]{{\textit{#1}}}
\usepackage{aliascnt}

\newtheorem{thm}{Theorem}[section]
\newtheorem{theorem}[thm]{Theorem}
\def\thmautorefname~#1\null{Theorem~#1\null}

\newaliascnt{lem}{thm}
\newtheorem{lemma}[lem]{Lemma}
\aliascntresetthe{lem}
\def\lemautorefname~#1\null{Lemma~#1\null}

\newaliascnt{prop}{thm}
\newtheorem{proposition}[prop]{Proposition}
\aliascntresetthe{prop}
\def\propautorefname~#1\null{Proposition~#1\null}

\newaliascnt{cor}{thm}
\newtheorem{corollary}[cor]{Corollary}
\aliascntresetthe{cor}
\def\corautorefname~#1\null{Corollary~#1\null}

\newaliascnt{conj}{thm}
\newtheorem{conjecture}[conj]{Conjecture}
\aliascntresetthe{conj}
\def\conjautorefname~#1\null{Conjecture~#1\null}

\theoremstyle{remark}
\newaliascnt{ex}{thm}
\newtheorem{example}[ex]{Example}

\def\exautorefname~#1\null{Example~#1\null}

\newtheorem{remark}[thm]{Remark}

\def\equationautorefname~#1\null{\text{(#1)\null}}
\def\figureautorefname~#1\null{Figure~#1\null}

\newcommand{\ie}{\textit{i.e.},~} 
\newcommand{\st}{:}
\newcommand{\bs}{\backslash}
\newcommand{\inv}{{-1}}
\renewcommand{\b}[1]{\mathbf{#1}} 
\newcommand{\closure}[1]{\overline{#1}}
\newcommand{\order}[1]{\left|#1\right|}

\newcommand{\set}[2]{\left\{ #1 \st #2 \right\}} 
\newcommand{\gen}[1]{\left< #1 \right>}
\newcommand{\genset}[2]{\gen{#1 \st #2}}
\renewcommand{\to}{\rightarrow} 

\newcommand{\Q}{\mathbb{Q}} 
\newcommand{\N}{\mathbb{N}} 
\newcommand{\mcA}{\mathcal{A}}
\newcommand{\mcP}{\mathcal{P}}
\newcommand{\msE}{\mathscr{E}}

\newcommand{\meet} {\wedge}
\newcommand{\bigmeet} {\bigwedge}
\newcommand{\join} {\vee}
\newcommand{\bigjoin} {\bigvee}
\newcommand{\precdot}{{\;<\!\!\!\!\cdot\;\,}} 
\newcommand{\cover}{\precdot\!\!}

\newcommand{\leqR}{\leq_{R}}
\newcommand{\geqR}{\geq_{R}}

\newcommand{\coverR}{\cover_{R}}
\newcommand{\joinR}{\join_{R}}
\newcommand{\meetR}{\meet_{R}}
\newcommand{\leqB}{\leq_{B}}

\newcommand{\longEl}{w_\circ}
\newcommand{\longElI}[1]{w_{\circ,#1}}

\DeclareMathOperator{\cl}{cl}

\DeclareMathOperator{\supp}{supp} 

\newcommand{\bclosure}[1]{\closure{#1}}

\newcommand{\bbclosure}[1]{\bnclosure{#1}{\infty}}
\newcommand{\bnclosure}[2]{\bclosure{#1}^{#2}}
\newcommand{\cclosure}[1]{\closure{#1}}
\newcommand{\cnclosure}[2]{\cclosure{#1}^{#2}}

\newcommand{\leslant}{\prec}
\renewcommand{\leqslant}{\preccurlyeq}
\newcommand{\initSects}{\mcA}
\newcommand{\twistedInv}[2]{#2 \cdot #1}
\newcommand{\symdiff}{\ominus}

\newcommand{\leqA}{\leq_A}

\makeatletter
\newtheorem*{rep@theorem}{\rep@title}
\newcommand{\newreptheorem}[2]{%
\newenvironment{rep#1}[1]{%
    \def\rep@title{\autoref{##1}}%
    \begin{rep@theorem}}%
    {\end{rep@theorem}}}
\makeatother

\newreptheorem{theorem}{Theorem}

\newcommand\re[1]{\color{red}{\b{#1}}}

\author{Aram Dermenjian}
\address{Universidad de Sevilla}
\email{aram.dermenjian.math@gmail.com}
\title{Bruhat Preclosure}


\begin{document}

\begin{abstract}
    In 2011, Dyer published a series of conjectures on the weak order of Coxeter groups.
    One of these conjectures stated that the inversion set of the join of two elements in a Coxeter group is equal to some ``closure'' of the union of their inversion sets.
    In this paper we show that this ``closure'' is in fact a preclosure, which we call the Bruhat preclosure, but \emph{is} a closure whenever our underlying set is an inversion set.
    By performing the Bruhat preclosure an infinite number of times we obtain a closure which we call the infinite Bruhat closure.
    We show in a uniform way that Dyer's conjecture is true when using the infinite Bruhat closure (instead of Bruhat preclosure) if the join exists between two elements.
    Finally, we end by showing in type $A$, the Bruhat preclosure is a closure thus giving a (second) proof that  Dyer's conjecture is true in type $A$.
\end{abstract}

\maketitle

\section{Introduction}
For an arbitrary Coxeter system $(W, S)$, there are two well-studied combinatorial structures on the elements of $W$: the weak order and the Bruhat graph.
For two elements $u$ and $v$ in $W$ the (right) weak order is given by $u \leqR v$ if and only if some reduced expression of $u$ is a \emph{prefix} of some reduced expression of $v$.
Alternatively, the weak order can be described through inversion sets $N(w)$ for $w \in W$ where $u \leqR v$ if and only if $N(u) \subseteq N(v)$.
The weak order is a meet-semilattice in general and is a lattice when $W$ is finite.
Letting $T$ be the set of (all) reflections, the Bruhat graph $\Omega$ is the labelled directed graph with vertex set $W$ and (labelled) directed edges $u \xrightarrow{t} v$ if there is some $t \in T$ such that $v = ut$ with $\ell(u) < \ell(v)$.
In \cite{Dyer_OnTheWeakOrderOfCoxeterGroups}, Dyer gives a conjectural relationship between the weak order and the Bruhat graph given in terms of an operator which we call the (right) Bruhat preclosure.
For a subset $A$ of $T$, let $\Omega_A$ be the subgraph of the Bruhat graph $\Omega$ restricted to the edges labelled by elements in $A$.
The \definition{(right) Bruhat preclosure} of $A$ is then the set of reflections $\bclosure{A} \subseteq T$ which can be reached in the subgraph $\Omega_A$ starting from the identity element of $W$.

Initially, the Bruhat preclosure was thought to be a closure, but \autoref{ex:H3_fail} shows a counterexample to the closure property.
Restricting to initial sections (a generalization of inversion sets defined in \autoref{ssec:initial-sections}) the Bruhat preclosure is a closure.
In fact, we prove something much stronger, showing that the closure of an initial section is itself.
\begin{reptheorem}{thm:Bruhat_is_closure}
    For $A$ an initial section, $\bclosure{A} = A$.
\end{reptheorem}
To rectify the fact that the Bruhat preclosure is not a closure, we extend the Bruhat preclosure into a closure by taking its infinite union $\bbclosure{A} = \cup_{n \in \N} \bnclosure{A}{n}$ where $\bnclosure{A}{n}$ represents the Bruhat preclosure being performed $n$ times.
We call this closure the \emph{infinite Bruhat closure}.
It turns out that the infinite Bruhat closure is a closure as it is of \emph{finite type} (defined in \autoref{sec:preclosure-extensions}).

The conjecture of Dyer \cite[Conjecture 2.8]{Dyer_OnTheWeakOrderOfCoxeterGroups} states that the Bruhat preclosure gives us information on the weak order.
In particular, Dyer's conjecture states that the Bruhat preclosure of the union of the inversion sets of two elements $u, v \in W$ is equal to the inversion set of their join: $N(u \joinR v) = \bclosure{N(u) \cup N(v)}$.
This article gives a positive result on a weaker form of Dyer's conjecture using the infinite Bruhat closure.
\begin{reptheorem}{thm:main}
    For $u$ and $v$ in $W$ such that $u \joinR v$ exists then 
    \[
        N(u \joinR v) = \bbclosure{N(u) \cup N(v)}.
    \]
\end{reptheorem}
In type $A$, it was shown by Biagioli and Perrone in \cite{Biagioli2025} that this can be strengthened, giving a positive result of Dyer's conjecture.
Furthermore, Biagioli and Perrone have mentioned \cite{Biagioli_Communication} that they also have a proof in type $B$ which is yet to be published.
In \autoref{sec:type-a-case} we show that the Bruhat preclosure is in fact a closure for type $A$, thus giving a second proof of Dyer's conjecture in this type.
Through many tests using Sage, we believe that Dyer's conjecture should be true for the Bruhat preclosure in general.
We discuss this in more detail in \autoref{sec:main-results}.

This article is organized in the following way.
We begin with some background details on the weak order, the Bruhat graph and reflection subgroups of Coxeter groups in \autoref{sec:the-weak-order}.
Then, we introduce our main object of study, the Bruhat preclosure, in \autoref{sec:bruhat-preclosure}.
We give examples of this preclosure and show that it is not a closure.
In \autoref{sec:twisted-bruhat-orders} we discuss twisted Bruhat graphs which are graphs used to understand the Bruhat preclosure. 
We study various properties on twisted Bruhat graphs and finish the section by proving \autoref{thm:Bruhat_is_closure}, that the closure of an initial section is itself.
Then in \autoref{sec:preclosure-extensions} we define preclosures being of finite type and show that preclosures of finite type extend naturally to a closure.
In particular, we show that the Bruhat preclosure is of finite type, thus allowing us to extend it to the infinite Bruhat closure.
In \autoref{sec:main-results} we prove the main result of this paper, \autoref{thm:main}, showing true a weaker form of Dyer's conjecture.
We end this paper by showing in \autoref{sec:type-a-case} that in type $A$ the Bruhat preclosure is a closure thus giving a second proof that Dyer's conjecture is true for type $A$.

\textbf{Acknowledgements}
The author would like to thank Matthew Dyer, Christophe Hohlweg, and Vincent Pilaud for multiple thoughtful and fruitful conversations on this topic.

\section{The weak order and the Bruhat graph}
\label{sec:the-weak-order}
In this section we give background and notation for the rest of the article.
Those with knowledge of the weak order and the Bruhat graph can safely skip this section.
For a more thorough background on Coxeter groups we direct the reader to the book by Bj\"orner and Brenti \cite{Bjorner_CombinatoricsofCoxeterGroups} and the book by Humphreys \cite{Humphreys_ReflectionGroupsandCoxeterGroups}.
For a background on posets we direct the reader to the book by Birkhoff \cite{Birkhoff_LatticeTheory}.

\subsection{Coxeter groups}
\label{ssec:coxeter-groups}
A \definition{Coxeter system} is a pair $(W,S)$ where $S$ is a finite set of \definition{simple reflections} and $W = \genset{S}{(st)^{m_{st}} = e\text{, }s,t\in S}$ is the \definition{Coxeter group} generated by $S$ such that $m_{st} \in \N \cup \left\{ \infty \right\}$ and where $m_{st} = 1$ if and only if $s = t$.
A Coxeter system is said to be \definition{finite} if $W$ is finite.
Furthermore, the set of \definition{reflections} is given by $T = \cup_{w \in W} wSw^{\inv}$.
For finite Coxeter groups there is a unique element of maximal length which we denote by $\longEl$ called the \definition{long element}.
Given a subset $I$ of $S$, let $W_I = \gen{I}$ be the \definition{(standard) parabolic subgroup} of $W$ associated to $I$.
If $W_I$ is finite, let $\longElI{I}$ denote its long element.


For an element $w \in W$ the \definition{length $\ell(w)$ of $w$} is the minimal length of all words representing $w$ as a product of the simple reflections $S$; in other words ${\ell(w) = \min\set{k}{w = s_1 s_2\ldots s_k,\, s_i \in S}}$.
If $w = s_1 s_2 \ldots s_n$ and $\ell(w) = n$ then we say $s_1 s_2 \ldots s_n$ is a \definition{reduced expression} or \definition{reduced word} for $w$.
The \definition{support} of a word is the minimal alphabet needed to produce the word: $\supp(s_1\ldots s_n) = \set{s_i}{i \in [n]}$.
It is well-known that the support of reduced expressions of elements of a Coxeter group all coincide.
Thus, we can define the \definition{support of an element} $\supp(w)$ to be the support of any of its reduced expressions.

Let $\mcP(T)$ denote the power set of $T$ and define the map $N : W \to \mcP(T)$ by $N(w) = \set{t \in T}{\ell(tw)< \ell(w)}$.
The set $N(w)$ is called the \definition{inversion set} of $w$.
It is known that $\order{N(w)} = \ell(w)$ and that inversion sets are \definition{reflection cocycles} (see \cite{Dyer_ReflectionSubgroupsOfCoxeterSystems}), \ie they satisfy the following two properties:
\begin{enumerate}
    \item $N(uv) = N(u) \symdiff uN(v)u^{\inv}$ for $u,v \in W$ and
    \item $N(s) = \left\{ s \right\}$ for $s \in S$
\end{enumerate}
where $\symdiff$ denotes the symmetric difference.

A \definition{reflection subgroup $W'$} is a subgroup of $W$ such that $W' = \gen{W' \cap T}$.
Letting the set of generators of $W'$ be denoted by $\chi(W') = \set{t \in T}{N(t) \cap W' = \left\{ t \right\}}$, then it is known that $(W', \chi(W'))$ is a Coxeter system, see for instance \cite[Theorem 3.3]{Dyer_ReflectionSubgroupsOfCoxeterSystems}.
If $\order{\chi(W')} = 2$ then $W'$ is said to be a \definition{dihedral reflection subgroup} of $W$.

\subsection{Posets}
A \definition{poset} $(P,\leqslant)$ is a set $P$ together with a partial order $\leqslant$.
Given two elements $x, y \in P$ the least upper bound of $x$ and $y$ (if it exists) is called the \definition{join} of $x$ and $y$ and is denoted $x \join y$.
The greatest lower bound of $x$ and $y$ (if it exists) is called the \definition{meet} and is denoted $x \meet y$.
If every pair of elements has a join then the poset is called a \definition{join-semilattice}.
Similarly, a \definition{meet-semilattice} is a poset where every two elements have a meet.
If a poset is a join-semilattice and a meet-semilattice, \ie every pair of elements has a join and a meet, then it is called a \definition{lattice}.

By abuse of notation, for $X = \left\{ x_1, x_2, \ldots, x_n \right\} \subseteq P$ we let $\bigjoin X$ denote $x_1 \join x_2 \join \cdots \join x_n$ and similarly $\bigmeet X = x_1 \meet x_2 \meet \cdots \meet x_n$.
We say $X$ is \definition{bounded above} by $y \in P$ if for all $x_i \in X$, $x_i \leqslant y$.
We define \definition{bounded below} similarly.

\subsection{Weak order and Bruhat graph}
The \definition{(right) weak order} is the partial order $\leqR$ on $W$ such that $u \leqR v$ if and only if $\ell(v) = \ell(u) + \ell(u^{\inv}v)$.
It is well-known that $u \leqR v$ if and only if $N(u) \subseteq N(v)$.
Additionally, $(W, \leqR)$ is a meet-semilattice (see \cite[Section 3.2]{Bjorner_CombinatoricsofCoxeterGroups}) and a lattice when $W$ is finite.

The \definition{Bruhat graph} is the edge-labelled directed graph $\Omega = (W, E)$ with vertex set $W$ and directed edges $E = \set{u \xrightarrow{t} v}{\ell(u) < \ell(v) \text{ and } v = ut \text{ for }\ t \in T}$.
The \definition{Bruhat order} $\leqB$ is then the partial order on $W$ such that $u \leqB v$ if and only if there exists a chain of directed edges in the Bruhat graph $(x_0, x_1), (x_1, x_2), \ldots, (x_{n-1},x_n)$ such that $x_0 = u$ and $x_n = v$.
Note that the Bruhat order is not a lattice in general.
Furthermore, $u \leqR v$ implies $u \leqB v$, but the reverse is not necessarily true.

\begin{example}
    \label{ex:A2-normal}
    Let $W$ by a type $A_2$ Coxeter group with $S = \{s, t\}$.
    Then the Bruhat graph of $W$ is the following graph:
    
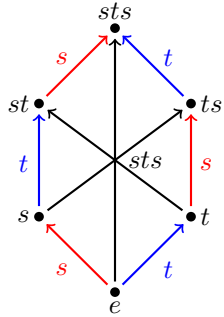
\begin{center}
        \begin{tikzpicture}
            [vertex/.style={inner sep=1pt,draw=black, circle, fill=black, thick},
                arr/.style={shorten >=4pt,shorten <=4pt,->,thick}
            ]
                \coordinate (e) at (0,0);
                \coordinate (s) at (-1,1);
                \coordinate (t) at (1,1);
                \coordinate (st) at (-1,2.5);
                \coordinate (ts) at (1,2.5);
                \coordinate (sts) at (0,3.5);

                \node[vertex] at (e) {};
                \node[vertex] at (s) {};
                \node[vertex] at (t) {};
                \node[vertex] at (st) {};
                \node[vertex] at (ts) {};
                \node[vertex] at (sts) {};

                \node[below] at (e) {$e$};
                \node[left] at (s) {$s$};
                \node[right] at (t) {$t$};
                \node[left] at (st) {$st$};
                \node[right] at (ts) {$ts$};
                \node[above] at (sts) {$sts$};

                \draw[arr, red] (e) -- (s);
                \draw[arr, blue] (e) -- (t);
                \draw[arr] (e) -- (sts);

                \draw[arr, blue] (s) -- (st);
                \draw[arr] (s) -- (ts);
                \draw[arr, red] (t) -- (ts);
                \draw[arr] (t) -- (st);

                \draw[arr, red] (st) -- (sts);
                \draw[arr, blue] (ts) -- (sts);

                \node[red] at (-0.7, 0.3) {$s$};
                \node[red] at (1.2, 1.7) {$s$};
                \node[red] at (-0.7, 3.1) {$s$};

                \node[blue] at (0.7, 0.3) {$t$};
                \node[blue] at (-1.2, 1.7) {$t$};
                \node[blue] at (0.7, 3.1) {$t$};

                \node at (0.4, 1.76) {$sts$};
        \end{tikzpicture}
    \end{center}
    Note that the labels for the edges show right multiplication.
    The arrows in red represent multiplying on the right by $s$, the arrows in blue represent multiplying on the right by $t$ and the black arrows in the middle represent multiplying on the right by $sts$.
\end{example}

\section{Bruhat Preclosure}
\label{sec:bruhat-preclosure}
In this section we define the main operator of this paper: the Bruhat preclosure.
Recall that given a set $Z$ a \definition{preclosure operator} is a map $cl : \mcP(Z) \to \mcP(Z)$ such that the following properties hold for $X, Y \subseteq Z$:
\begin{enumerate}
    \item $X \subseteq cl(X)$ and
    \item $X \subseteq Y$ implies $cl(X) \subseteq cl(Y)$
\end{enumerate}
If furthermore $cl(X) = cl(cl(X))$, then $cl$ is called a \definition{closure operator}.

Recall that the Bruhat graph $\Omega = (W,E)$ is the edge-labelled directed graph whose labels are given by multiplication on the right by reflections.
For $A \subseteq T$, a \definition{right $A$-path} is a path in $\Omega$ traversing only edges labelled by elements in $A$.
Alternatively, letting $\Omega_A$ be the subgraph of $\Omega$ restricted to edges labelled by elements in $A$, then a right $A$-path is a path in $\Omega_A$.
A left $A$-path can be similarly defined if the edges in $\Omega$ are labelled by left reflections instead of right.
In this article we exclusively use right reflections as labels and thus will refer to a right $A$-path simply as an \definition{$A$-path}.
\begin{example}
    Let $W$ be a type $A_2$ Coxeter group with simple reflections $S = \left\{ s,t \right\}$ and reflections $T = \left\{ s,t,sts \right\}$.
    Letting $A = \left\{ s, {sts} \right\} \subseteq T$ then the following path in $\Omega$ is an $A$-path
    \[
        e \xrightarrow{s} s \xrightarrow{sts} ts,
    \]
    but the following two paths are not
    \[
        e \xrightarrow{t} t
        \quad\text{and}\quad
        e \xrightarrow{sts} sts \xrightarrow{s} st.
    \]
    The first one is not an $A$-path since $t \notin A$ and the second is not an $A$-path since $sts \xrightarrow{s} st$ is not an edge in the Bruhat graph as $\ell(sts) > \ell(st)$.
\end{example}

For $A \subseteq T$ the \definition{(right) Bruhat preclosure} is the preclosure given by
\[
    \bclosure{A} = \set{t \in T}{\text{there exists a right $A$-path from $e$ to $t$}}.
\]
The left Bruhat preclosure is defined similarly.
As before we will say \definition{Bruhat preclosure} to mean right Bruhat preclosure.

It is readily seen by the construction of the Bruhat preclosure that it is in fact a preclosure.
We give two examples of the Bruhat preclosure.

\begin{example}\label{ex:A3_ex}
    We first look at an example of the Bruhat preclosure for a Coxeter group of type $A_3$ with presentation
    \[
        W = \genset{r,s,t}{r^2 = s^2 = t^2 = (rs)^3 = (st)^3 = (rt)^2 = e}.
    \]
    Let $A = N(rts) = \left\{ r, rstsr, t \right\}$.
    The following graph shows all $A$-paths which start at the identity.
    The edges in this graph are directed upward and are labelled with right multiplication.
    
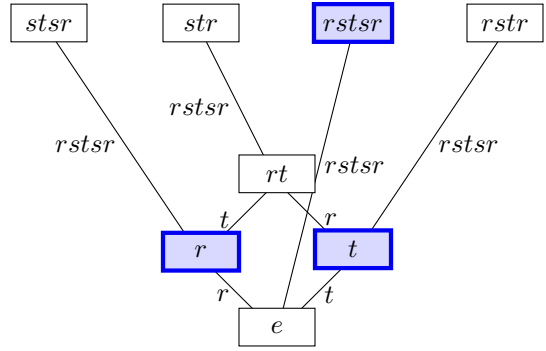
\begin{center}
        \begin{tikzpicture}
            [vertex/.style={draw=black, anchor=base, rectangle, fill=white, minimum width=1cm, minimum height=0.5cm},
                vertexRef/.style={draw=blue!95!black, anchor=base, rectangle, fill=blue!15!white, minimum width=1cm, minimum height=0.5cm, ultra thick}
            ]

            \coordinate (e) at (0,0);
            \coordinate (r) at (-1,1);
            \coordinate (t) at (1,1);
            \coordinate (rt) at (0,2);
            \coordinate (str) at (-1,4);
            \coordinate (stsr) at (-3,4);
            \coordinate (rstr) at (3,4); 
            \coordinate (rstsr) at (1,4); 
            
            \draw[->] (e) -- (r.south) node [midway, left] {$r$};
            \draw[->] (e) -- (t) node [midway, right] {$t$};
            \draw[->] (e) -- (rstsr) node [midway, above right] {$rstsr$};
            \draw[->] (r) -- (rt) node [midway, left] {$t$};
            \draw[->] (t) -- (rt) node [midway, right] {$r$};
            \draw[->] (r) -- (stsr) node [midway,left] {$rstsr$};
            \draw[->] (t) -- (rstr) node [midway,right] {$rstsr$};
            \draw[->] (rt) -- (str) node [midway,left] {$rstsr$};

            \node[vertex] at (e) {$e$};
            \node[vertexRef] at (r) {$r$};
            \node[vertexRef] at (t) {$t$};
            \node[vertex] at (rt) {$rt$};
            \node[vertex] at (str) {$str$};
            \node[vertex] at (stsr) {$stsr$};
            \node[vertex] at (rstr) {$rstr$};
            \node[vertexRef] at (rstsr) {$rstsr$};

        \end{tikzpicture}
    \end{center}
    The reflections have been marked in blue with their border thickened.
    This shows $\bclosure{A} = \left\{ r,rstsr, t \right\} = A$.
    The reflections reached by traversing this subgraph are exactly the same reflections we started with.
    As we started with an inversion set, this leads us to conjecture that the Bruhat preclosure of an inversion set is the inversion set itself.
    In fact, this turns out to be true!
    This fact is shown in \autoref{thm:Bruhat_is_closure}.
    Although we'd like this to be true for an arbitrary $A \subseteq T$, we have the following counter-example.
\end{example}
\begin{example}\label{ex:H3_fail}
    Let $W$ be a type $H_3$ Coxeter group with $S = \left\{ r,s,t \right\}$ such that $(rs)^5 = (st)^3 = (rt)^2 = e$.
    We will give an $A \subseteq T$ such that $A \subsetneq \bclosure{A} \subsetneq \bclosure{\bclosure{A}}$ showing that the Bruhat preclosure is not a closure for arbitrary subsets.
    With that in mind, let 
    \[
        A = \left\{ r,\, srs,\, tsrsrst,\, strstrstrstrs\right\}.
    \]
    After tracing the Bruhat graph (see \autoref{fig:H3-fail}), one gets that
    \begin{align*}
        \bclosure{A} &= \left\{ r,\, srs,\, tsrsrst,\, strstrstrstrs,\, \mathbf{srsrs} \right\} \text{ and}\\
        \bclosure{\bclosure{A}} &= \left\{ r,\, srs,\, tsrsrst,\, strstrstrstrs,\, srsrs,\, \mathbf{rstrsrstrsr} \right\}
    \end{align*}
    where the new elements added are bolded.

    We see this in the (subgraph of the) Bruhat graph in \autoref{fig:H3-fail} where the reflections in $A$ have a thickened border in black, the additional element in $\bclosure{A}$ is in blue and the additional element in $\bclosure{\bclosure{A}}$ is in red.
    The edges labelled by elements in $A$ are black and the edges labelled by $srsrs$ are thickened, dashed and in blue.
    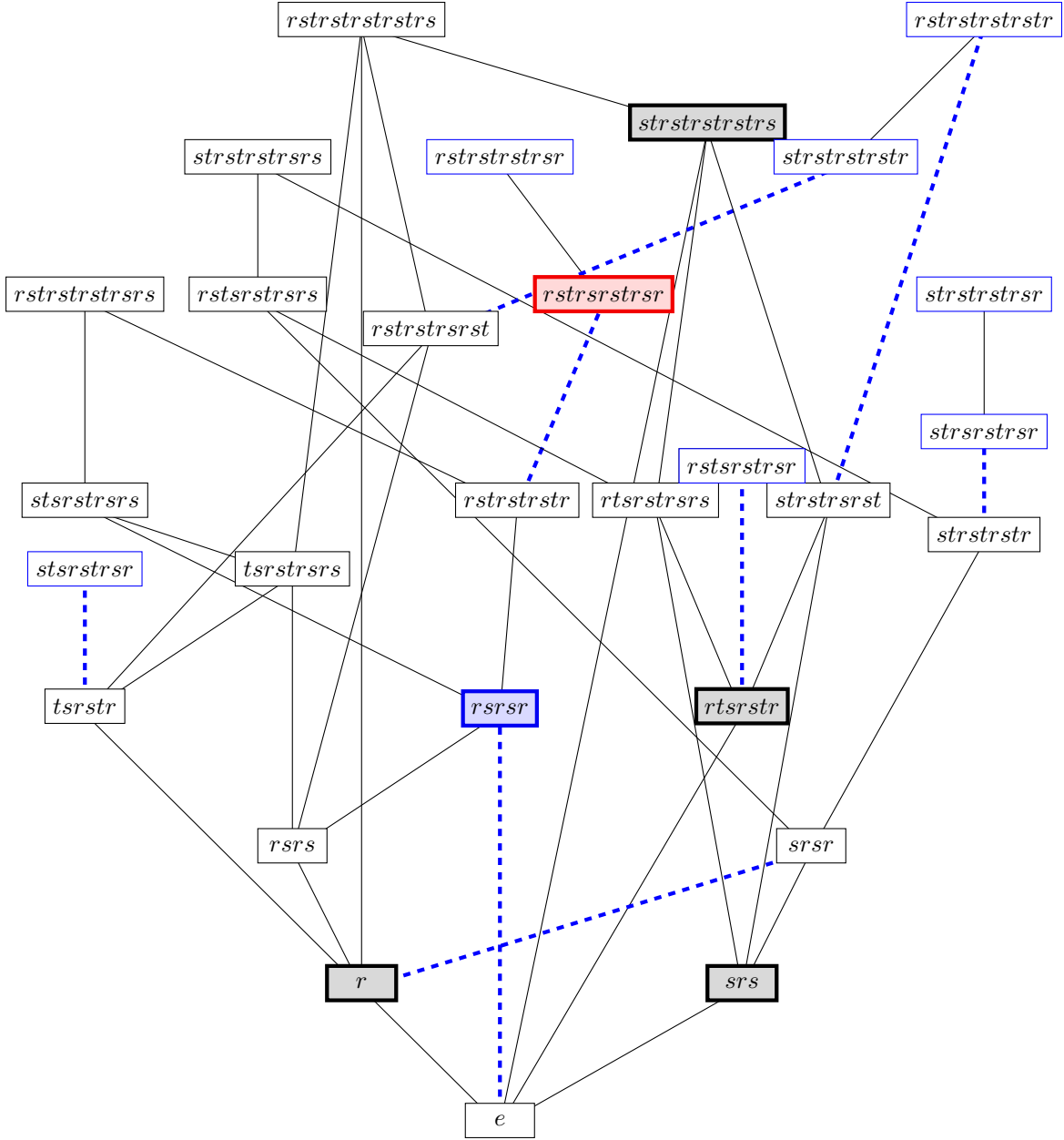
\begin{figure}
        \input{figures/H3.tex}
        \caption{The subgraph of the Bruhat graph for the type $H_3$ Coxeter group where the black edges are labelled by reflections shaded black and the blue dashed edges are labelled by the reflection shaded blue.}
        \label{fig:H3-fail}
    \end{figure}
\end{example}

We leave as an open problem when the Bruhat preclosure is in fact a closure.
From the previous example, one can deduce any Coxeter system with $H_3$ as a subsystem does not produce a closure (such as $H_4$).
\autoref{thm:type_A_closure} on the other hand will show that in type $A$, the Bruhat preclosure is a closure.
Understanding the closure structure better would allow us to solve Dyer's conjecture more precisely.
Computer testing in types $B_n$ (up to $n = 4$) and $D_n$ (up to $n = 5$) seem to imply that the Bruhat preclosure is a closure in these types.
In type $F_4$ we are able to produce a counter-example through computer testing:
\begin{align*}
    A &= \{utstu,sts,usrtsrtsu,utu,u,tsrutsrtsut,rstsrutsrtsutsr\} \\
    \bclosure{A} &= \{utstu,sts,usrtsrtsu,utu,u,tsrutsrtsut,rstsrutsrtsutsr, ustsu,tsutsut\} \\
    \bclosure{\bclosure{A}} &= \{utstu,sts,usrtsrtsu,utu,u,tsrutsrtsut,rstsrutsrtsutsr,ustsu,tsutsut,utsrtsutsrutstu\}
\end{align*}
where $(rs)^3 = (st)^4 = (tu)^3 = e$ and all other simple reflection pairs commute.
Our intuition leads us to believe that the Bruhat preclosure is only a closure whenever $m_{st} = 3$ for all $s, t \in S$, in other words, only in the simply laced case (types $A$, $D$ and $E$).
We do not have a specific reason for this other than the facts that types $H$ and $F$ fail, but type $A$ is successful.
Thus, we conjecture the following:
\begin{conjecture}
    Let $W$ be a type $A$, $D$, or $E$ Coxeter group and let $A \subseteq T$.
    Then $\bclosure{A} = \bclosure{\bclosure{A}}$.
\end{conjecture}
This conjecture would imply both that the Bruhat preclosure is a closure in these types and that Dyer's conjecture is true (using \autoref{thm:main}).
As of yet, we have not been able to show these results.

\section{Twisted Bruhat Orders}
\label{sec:twisted-bruhat-orders}
In this section we show that \autoref{ex:A3_ex} is true in general: \ie that the Bruhat preclosure of an inversion set is itself.
This will be done using twisted Bruhat orders which were first studied by Dyer in \cite{Dyer_HeckeAlgebrasAndShellingsOfBruhatIntervals} and is a way of reversing particular edges in the Bruhat graph to get a twisted version of the Bruhat graph.
They are a sort of generalization of the Bruhat order in Coxeter groups as described in \autoref{ssec:twisted-bruhat-graphs}.

\subsection{Initial sections}\label{ssec:initial-sections}
A \definition{reflection order} on the set of reflections $T$ is a total order $\leqslant$ such that for any dihedral reflection subgroup $W'$ of $W$ with $\chi(W') = \left\{ r,s \right\}$ either $r \leslant rsr \leslant \cdots \leslant srs \leslant s$ or $s \leslant srs \leslant \cdots \leslant rsr \leslant r$.
It is an amazing result that reflection orders exist on Coxeter systems (see \cite[Proposition 2.5]{Dyer_HeckeAlgebrasAndShellingsOfBruhatIntervals}).
We say that a subset $A$ of $T$ is an \definition{initial section} if there exists a reflection order $\leqslant$ on $T$ such that $a \leslant b$ for all $a \in A$ and $b \in T\bs A$.
We let $\initSects$ denote the set of all initial sections.
In fact, the finite initial sections are inversion sets.

\begin{proposition}[{\cite[Lemma 2.11]{Dyer_HeckeAlgebrasAndShellingsOfBruhatIntervals}}]
    \label{prop:finite_inits_are_inversions}
    Let $A$ be a finite subset of $T$.
    The following are equivalent:
    \begin{enumerate}
        \item $A \in \initSects$
        \item $A = N(w)$ for some $w \in W$.
    \end{enumerate}
\end{proposition}

This means for finite $W$ every inversion set is an initial section of some reflection order and vice-versa.
\begin{example}
    Let us look at the Coxeter group of type $A_3$ with $S = \left\{ r,s,t \right\}$ such that $(rs)^3 = (st)^3 = (rt)^2 = e$.
    Recall that the set of reflections for $A_3$ is given by $T = \left\{ r, s, t, rsr, sts, rstsr \right\}$.
    The following is then a reflection order on $A_3$:
    \[
        r \leslant t \leslant rstsr \leslant sts \leslant rsr \leslant s.
    \]
    Notice that taking any initial section of the reflection gives us an inversion set.
    As an example, the initial section
    \[
        r \leslant t \leslant rstsr
    \]
    is associated to the inversion set $N(rts) = \left\{ r, t, rstsr \right\}$ as in \autoref{ex:A3_ex}.

    It should be noted that not every inversion set is an initial section of this particular reflection order.
    Other reflection orders are needed in order to find the other inversion sets.
    In particular, in the finite case, the set of all reduced expressions of the long element gives the reflection orders necessary to get all inversion sets.
\end{example}

\subsection{Twisted Bruhat graphs}\label{ssec:twisted-bruhat-graphs}
We use these initial sections to define partial orders on $W$ called twisted Bruhat orders. 

Given a subset $A$ of $T$ we take the edge set $E$ of $\Omega$ and twist all edges labelled with elements in $A$ to reverse their direction.
This new graph is called the \definition{twisted Bruhat graph of $A$} and is denoted $\Omega_{(W,A)}$ where the vertex set is $W$ and the set of directed edges is given by
\begin{align*}
    E_{(W,A)} &= \begin{cases*}
        (u, ut) & for $t \notin A$ and $(u,ut) \in E$\\
        (ut, u) & for $t \in A$ and $(u,ut) \in E$
    \end{cases*}\\
    &= \begin{cases*}
        u \xrightarrow{t} ut & for $t \notin A$ and $(u,ut) \in E$\\
        ut \xrightarrow{t} u & for $t \in A$ and $(u,ut) \in E$.
    \end{cases*}
\end{align*}

\begin{remark}
    By comparing the edges, we note that the subgraph of $\Omega_{(W, A)}$ where we restrict to edges labelled by $t \in A$ is precisely the subgraph $\Omega_A$ (where we restrict $\Omega$ to edges labelled by $A$) with \emph{all} arrows reversed.
    In other words, there is a right $A$-path from $e$ to $w$ if and only if there exists a path from $w$ to $e$ in $\Omega_{(W, A)}$.
\end{remark}

Another method of describing the edge set $E_{(W, A)}$ is through a twisting of inversion sets.
For $w \in W$ and $A \subseteq T$ the \definition{twisted inversion set of $w$ by $A$} is the $W$-action on $\mcP(T)$ given by the reflection cocycle
\[
    \twistedInv{A}{w} = N(w) \symdiff wAw^{\inv}.
\]
Note that for $u \in W$ we have $\twistedInv{N(u)}{w} = N(wu)$.
Furthermore, if $A = \varnothing$ then the twisted inversion set is the standard inversion set, \ie $w \cdot \varnothing = N(w)$.

Twisted inversion sets behave in a very similar way to standard inversion sets.
\begin{lemma}
    \label{lem:a_notin_wA_iff_a_in_twA}
    For $A \subseteq T$,  $w \in W$ and $t \in T$ then $t \notin \twistedInv{A}{w}$ if and only if $t \in tw \cdot A$.
\end{lemma}
\begin{proof}
    By definition $t \notin \twistedInv{A}{w}$ if and only if $t \notin N(w) \symdiff wAw^{\inv}$.
    Or in other words if and only if $t \notin tN(w)t^{\inv} \symdiff twAw^{\inv} t^{\inv}$.
    Since $t \in N(t)$ and $\symdiff$ is associative we have 
    \begin{align*}
        t \in &(N(t) \symdiff tN(w)t^{\inv}) \symdiff tw(A)w^{\inv} t^{\inv} \\
            &= N(tw) \symdiff twAw^{\inv} t^{\inv} \\
            &= tw \cdot A
    \end{align*}
    as desired.
\end{proof}

These twisted inversion sets can then be used to characterize the edges of a twisted Bruhat graph.
\begin{proposition}\label{prop:change_arrows}
    For $A \subseteq T$ and $u,v \in W$ such that $v = ut = t'u$ then $(u,v) \in E_{(W,A)}$ if and only if $t' \notin \twistedInv{A}{u}$.
\end{proposition}
\begin{proof}
    We recall the following fact (since $ut = t'u$)
    \[
        t \in A \text{ if and only if } t' \in uAu^{\inv} \tag{$\star$}\label{eq:ca_2}
    \]
    and the definition of an inversion set
    \[
        \ell(u) > \ell(t'u) = \ell(v) \tag{$\star\star$}\label{eq:ca_1}
        \text{ if and only if }
        t' \in N(u).
    \]

    Suppose $(u,v) \in E_{(W,A)}$.
    By definition this is true if and only if $t \notin A$ and $\ell(u) < \ell(ut)$ or $t \in A$ and $\ell(u) > \ell(ut)$.
    But by \autoref{eq:ca_2} and \autoref{eq:ca_1} this is true if and only if $t'$ is neither in $uAu^{\inv}$ nor in $N(u)$ or in both.
    In other words, if and only if $t' \notin N(u) \symdiff uAu^{\inv} = u \cdot A$ as desired.

\end{proof}

Therefore, the edge set of a twisted Bruhat graph can be described as:
\begin{equation}
    \label{eq:twisted-graph}
    E_{(W,A)} = \set{(w,tw)}{w \in W,\, t \notin \twistedInv{A}{w}} = \set{(tw,w)}{w \in W,\, t \in \twistedInv{A}{w}}
\end{equation}
where the second equality uses \autoref{lem:a_notin_wA_iff_a_in_twA}.

We are now ready to give a definition of a twisted Bruhat order using the twisted Bruhat graph.
Let $C_n^A(u,v)$ be the set
\[
    C_n^A(u,v) = \set{(w_0, w_1, \ldots, w_n) \in W^{n+1}}{(w_{i-1},w_i) \in E_{(W,A)},\, w_0 = u,\, w_n = v}.
\]
Given an initial section $A \in \initSects$, the \definition{twisted Bruhat order} $\leqA$ is the partial order on $W$ where $u \leqA v$ if and only if $C_n^A(u,v) \neq \emptyset$ for some $n \in \N$.
It was shown in \cite[Theorem 8.1.4]{Edgar_DominanceAndRegularityInCoxeterGroups} that this is a partial order if and only if $A$ is an initial section.
Note that if $A$ is not an initial section a twisted Bruhat order might not exist, but the twisted Bruhat graph can still be defined.

\begin{example}\label{ex:A2}
    We use a Coxeter group of type $A_2$ to give an example of a twisted Bruhat order and an anti-example for when we don't have an initial section.
    The (untwisted) Bruhat graph is given in \autoref{ex:A2-normal}.
    First let $S = \left\{ s,t \right\}$, $A = N(st) = \left\{ s, sts \right\}$ and $A' = \left\{ s,t \right\}$.
    Then recalling our definition of $E_{(W,A)}$ we have the following two figures respectively.
    
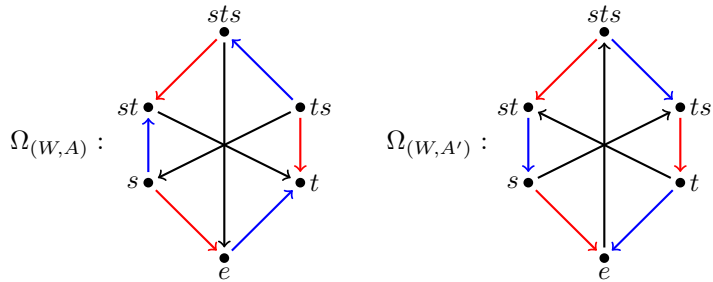
\begin{center}
        \begin{figure}[H]
            \begin{tikzpicture}
                [vertex/.style={inner sep=1pt,draw=black, circle, fill=black, thick},
                    arr/.style={shorten >=4pt,shorten <=4pt,->,thick}
                ]
                \begin{scope}
                    \coordinate (e) at (0,0);
                    \coordinate (s) at (-1,1);
                    \coordinate (t) at (1,1);
                    \coordinate (st) at (-1,2);
                    \coordinate (ts) at (1,2);
                    \coordinate (sts) at (0,3);

                    \node[vertex] at (e) {};
                    \node[vertex] at (s) {};
                    \node[vertex] at (t) {};
                    \node[vertex] at (st) {};
                    \node[vertex] at (ts) {};
                    \node[vertex] at (sts) {};

                    \node at (-2.2,1.5) {$\Omega_{(W,A)}:$};
                    \node[below] at (e) {$e$};
                    \node[left] at (s) {$s$};
                    \node[right] at (t) {$t$};
                    \node[left] at (st) {$st$};
                    \node[right] at (ts) {$ts$};
                    \node[above] at (sts) {$sts$};

                    \draw[arr, red] (s) -- (e);
                    \draw[arr, blue] (e) -- (t);
                    \draw[arr] (sts) -- (e);

                    \draw[arr, blue] (s) -- (st);
                    \draw[arr] (ts) -- (s);
                    \draw[arr, red] (ts) -- (t);
                    \draw[arr] (st) -- (t);

                    \draw[arr, red] (sts) -- (st);
                    \draw[arr, blue] (ts) -- (sts);
                \end{scope}
                \begin{scope}[shift={(5,0)}]
                    \coordinate (e) at (0,0);
                    \coordinate (s) at (-1,1);
                    \coordinate (t) at (1,1);
                    \coordinate (st) at (-1,2);
                    \coordinate (ts) at (1,2);
                    \coordinate (sts) at (0,3);

                    \node[vertex] at (e) {};
                    \node[vertex] at (s) {};
                    \node[vertex] at (t) {};
                    \node[vertex] at (st) {};
                    \node[vertex] at (ts) {};
                    \node[vertex] at (sts) {};

                    \node at (-2.2,1.5) {$\Omega_{(W,A')}:$};
                    \node[below] at (e) {$e$};
                    \node[left] at (s) {$s$};
                    \node[right] at (t) {$t$};
                    \node[left] at (st) {$st$};
                    \node[right] at (ts) {$ts$};
                    \node[above] at (sts) {$sts$};

                    \draw[arr, red] (s) -- (e);
                    \draw[arr, blue] (t) -- (e);
                    \draw[arr] (e) -- (sts);

                    \draw[arr, blue] (st) -- (s);
                    \draw[arr] (s) -- (ts);
                    \draw[arr, red] (ts) -- (t);
                    \draw[arr] (t) -- (st);

                    \draw[arr, red] (sts) -- (st);
                    \draw[arr, blue] (sts) -- (ts);
                \end{scope}
            \end{tikzpicture}
        \end{figure}
    \end{center}

    Notice that $A$ is an initial section (as it is an inversion set) and our graph gives us an order with $ts$ as the bottom element and $t$ as the top element.
    On the other hand, $A'$ is not an initial section and no partial order exists.
    In particular, we have the cycle $e \to sts \to ts \to t \to e$ in the graph and cycles cannot occur in partial orders.
\end{example}

In the case that $A$ is a finite initial section the twisted Bruhat order is nothing more than a poset isomorphism.
More explicitly, if $A = N(w)$ for some $w \in W$ then we have a poset isomorphism $x \mapsto xw^{\inv}$.
In the previous examples we saw this with the case $w = st$.
The identity was sent to $ts$ making $ts$ the bottom element.
Similarly, the top element $\longEl = sts$ was sent to $(sts)(ts) = t$.

This isomorphism makes most of the observations here trivial in some sense by using the poset isomorphism for the results.
The interesting case comes when $A$ is not finite and we no longer get a poset isomorphism as seen in the following example.
\begin{example}
    \label{ex:Ainf}
    Let $W$ be the infinite Coxeter group with $S = \left\{ s,t \right\}$ such that $m_{st} = \infty$, \ie there is no relation between $s$ and $t$.
    The Bruhat graph for this is given in the figure below on the left.
    As there are only two simple reflections, there are only two reflection orders:
    \[
        s \leslant sts \leslant ststs \leslant \ldots \leslant tstst \leslant tst \leslant t
    \]
    and its reverse.
    We let $A$ be the initial section of all reflections that start with $s$.
    In other words, $A = \left\{ s, sts, ststs, \ldots \right\} = \set{r \in T}{\ell(sr) < \ell(r)}$.
    This set is infinite and is therefore not associated to any inversion set.
    The twisted Bruhat order associated to $A$ is given in the figure on the right.
    Notice that there is in fact no bottom element!
    This fact will be even more evident after the next section.
    Therefore, there does not exist a poset isomorphism between these two posets:

    
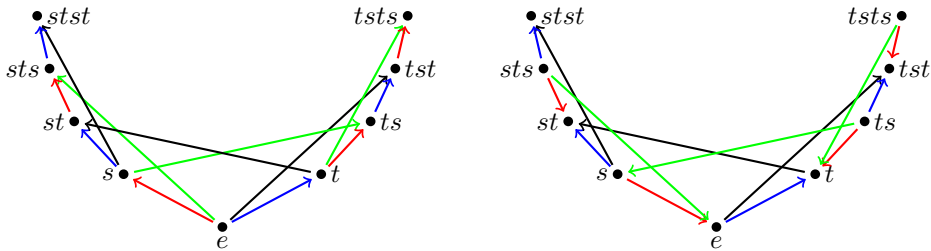
\begin{center}
        \begin{tikzpicture}
                [vertex/.style={inner sep=1pt,draw=black, circle, fill=black, thick},
                    arr/.style={shorten >=4pt,shorten <=4pt,->,thick},
                    sarr/.style={color=red},
                    tarr/.style={color=blue},
                    stsarr/.style={color=green},
                    tstarr/.style={},
                    ststsarr/.style={},
                    tststarr/.style={},
                    xscale=0.65,
                    yscale=0.7
                ]
                \begin{scope}[shift={(-3,0)}]
                    \coordinate (e) at (0,0);
                    \coordinate (s) at (-2,1);
                    \coordinate (t) at (2,1);
                    \coordinate (st) at (-3,2);
                    \coordinate (ts) at (3,2);
                    \coordinate (sts) at (-3.5,3);
                    \coordinate (tst) at (3.5,3);
                    \coordinate (stst) at (-3.75,4);
                    \coordinate (tsts) at (3.75,4);

                    \node[vertex] at (e) {};
                    \node[vertex] at (s) {};
                    \node[vertex] at (t) {};
                    \node[vertex] at (st) {};
                    \node[vertex] at (ts) {};
                    \node[vertex] at (sts) {};
                    \node[vertex] at (tst) {};
                    \node[vertex] at (stst) {};
                    \node[vertex] at (tsts) {};

                    \node[below] at (e) {$e$};
                    \node[left] at (s) {$s$};
                    \node[right] at (t) {$t$};
                    \node[left] at (st) {$st$};
                    \node[right] at (ts) {$ts$};
                    \node[left] at (sts) {$sts$};
                    \node[right] at (tst) {$tst$};
                    \node[right] at (stst) {$stst$};
                    \node[left] at (tsts) {$tsts$};

                    \draw[arr,sarr] (e) -- (s);
                    \draw[arr,tarr] (e) -- (t);
                    \draw[arr,stsarr] (e) -- (sts);
                    \draw[arr,tstarr] (e) -- (tst);

                    \draw[arr,tarr] (s) -- (st);
                    \draw[arr,stsarr] (s) -- (ts);
                    \draw[arr,tstarr] (s) -- (stst);

                    \draw[arr,sarr] (t) -- (ts);
                    \draw[arr,stsarr] (t) -- (tsts);
                    \draw[arr,tstarr] (t) -- (st);

                    \draw[arr,sarr] (st) -- (sts);

                    \draw[arr,tarr] (ts) -- (tst);

                    \draw[arr,tarr] (sts) -- (stst);

                    \draw[arr,sarr] (tst) -- (tsts);


                \end{scope}
                \begin{scope}[shift={(7,0)}]
                    \coordinate (e) at (0,0);
                    \coordinate (s) at (-2,1);
                    \coordinate (t) at (2,1);
                    \coordinate (st) at (-3,2);
                    \coordinate (ts) at (3,2);
                    \coordinate (sts) at (-3.5,3);
                    \coordinate (tst) at (3.5,3);
                    \coordinate (stst) at (-3.75,4);
                    \coordinate (tsts) at (3.75,4);

                    \node[vertex] at (e) {};
                    \node[vertex] at (s) {};
                    \node[vertex] at (t) {};
                    \node[vertex] at (st) {};
                    \node[vertex] at (ts) {};
                    \node[vertex] at (sts) {};
                    \node[vertex] at (tst) {};
                    \node[vertex] at (stst) {};
                    \node[vertex] at (tsts) {};

                    \node[below] at (e) {$e$};
                    \node[left] at (s) {$s$};
                    \node[right] at (t) {$t$};
                    \node[left] at (st) {$st$};
                    \node[right] at (ts) {$ts$};
                    \node[left] at (sts) {$sts$};
                    \node[right] at (tst) {$tst$};
                    \node[right] at (stst) {$stst$};
                    \node[left] at (tsts) {$tsts$};

                    \draw[arr,tarr] (e) -- (t);
                    \draw[arr,tstarr] (e) -- (tst);

                    \draw[arr,sarr] (s) -- (e);
                    \draw[arr,tarr] (s) -- (st);
                    \draw[arr,tstarr] (s) -- (stst);

                    \draw[arr,tstarr] (t) -- (st);


                    \draw[arr,sarr] (ts) -- (t);
                    \draw[arr,tarr] (ts) -- (tst);
                    \draw[arr,stsarr] (ts) -- (s);

                    \draw[arr,sarr] (sts) -- (st);
                    \draw[arr,tarr] (sts) -- (stst);
                    \draw[arr,stsarr] (sts) -- (e);

                    \draw[arr,sarr] (tsts) -- (tst);
                    \draw[arr,stsarr] (tsts) -- (t);
                \end{scope}
        \end{tikzpicture}
    \end{center}
    In the above example we have colored the edges red for $s$, blue for $t$, green for $sts$ and black for $tst$.
\end{example}

\subsection{Twisted length}
Just as with the standard Bruhat order, a useful property associated to twisted Bruhat order is a length function.
As we did with inversion sets we generalize the standard length function for subsets of $T$.
For $A \subseteq T$ and $v,w \in W$ we define the \definition{twisted length function from $v$ to $w$ with respect to $A$} as the function given by
\[
    \ell_A(v,w) = \ell(wv^{\inv}) - 2 \order{N(vw^{\inv}) \cap v \cdot A}.
\]
For ease of notation we let $\ell_A(w) = \ell_A(e,w)$ where $e$ is the identity element in $W$.
Then, we have that$e \cdot A = N(e) \symdiff eAe^{\inv} = A$ implies $\ell_A(w) = \ell(w) - 2 \order{N(w^{\inv}) \cap A}$.
As expected, if $A = \varnothing$ then we get the standard length function, $\ell_{\varnothing}(w) = \ell(w) - 2\order{N(w^{\inv}) \cap e \cdot \varnothing} = \ell(w)$ since $e \cdot \varnothing = N(e) = \varnothing$.

\begin{example}
    We give examples of twisted length functions for the Coxeter group of type $A_2$.
    Let $S = \left\{ s,t \right\}$, $A = N(st) = \left\{ s,sts \right\}$ and $A' = \left\{ s,t \right\}$ as in \autoref{ex:A2}.
    We end up with the following twisted lengths:
    \begin{center}
        \begin{tabular}{r || c | c}
            $w$     & $\ell_A(w)$ & $\ell_{A'}(w)$ \\
            \hline
            \hline
            $e$     & $0$   & $0$   \\
            $s$     & $-1$  & $-1$  \\
            $t$     & $1$   & $-1$  \\
            $st$    & $0$   & $0$   \\
            $ts$    & $-2$  & $0$   \\
            $sts$   & $-1$  & $-1$  \\
        \end{tabular}
    \end{center}
    Note that although $A'$ is not an initial section, we can still use the length function.
\end{example}
\begin{example}
    We give an example of a twisted length function for an infinite Coxeter group continuing from \autoref{ex:Ainf} where $W$ is the infinite Coxeter group with $S = \left\{ s,t \right\}$ and where $A = \left\{ s, sts, ststs, \ldots \right\}$.
    
    Since everything is dependent on $\order{N(w^{\inv}) \cap A}$ then the twisted length function will be dependent on whether a word ends in $s$ or not.
    If $\ell(ws) < \ell(w)$ then $N(w^{\inv}) \cap A = N(w^{\inv}) = \ell(w)$ since $w^{\inv}$ begins with $s$ making all reflections in $N(w^{\inv})$ start with $s$.
    Likewise, if $\ell(ws) > \ell(w)$ then $N(w^{\inv}) \cap A = \varnothing$ as no reflections in $N(w^{\inv})$ will start with $s$.
    Therefore, we have $\ell_A(w) = \pm \ell(w)$ depending on the last letter in a reduced expression for $w$.
    Here are a few of these lengths:
    \begin{center}
        \begin{tabular}{r || c}
            $w$     & $\ell_A(w)$ \\
            \hline
            \hline
            $e$     & $0$   \\
            $s$     & $-1$  \\
            $ts$    & $-2$  \\
            $sts$   & $-3$  \\
        \end{tabular}
        $\qquad$
        \begin{tabular}{r || c}
            $w$     & $\ell_A(w)$ \\
            \hline
            \hline
            $t$     & $1$   \\
            $st$    & $2$   \\
            $tst$   & $3$   \\
            $stst$  & $4$   \\
        \end{tabular}
    \end{center}

    This allows us to redraw the twisted Bruhat graph of \autoref{ex:Ainf} as a total order, further solidifying the fact that no poset isomorphism exists between the twisted Bruhat poset associated to $A$ and the Bruhat poset.

    
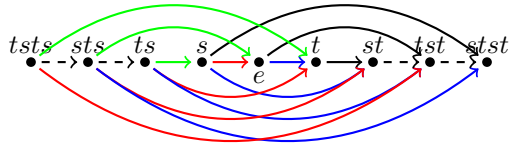
\begin{center}
        \begin{tikzpicture}
                [vertex/.style={inner sep=1pt,draw=black, circle, fill=black, thick},
                    arr/.style={shorten >=4pt,shorten <=4pt,->,thick},
                    sarr/.style={color=red},
                    tarr/.style={color=blue},
                    stsarr/.style={color=green},
                    tstarr/.style={},
                    ststsarr/.style={},
                    tststarr/.style={},
                    darr/.style={dashed},
                    xscale=0.75
                ]
                    \coordinate (e) at (0,0);
                    \coordinate (s) at (-1,0);
                    \coordinate (t) at (1,0);
                    \coordinate (st) at (2,0);
                    \coordinate (ts) at (-2,0);
                    \coordinate (sts) at (-3,0);
                    \coordinate (tst) at (3,0);
                    \coordinate (stst) at (4,0);
                    \coordinate (tsts) at (-4,0);

                    \node[vertex] at (e) {};
                    \node[vertex] at (s) {};
                    \node[vertex] at (t) {};
                    \node[vertex] at (st) {};
                    \node[vertex] at (ts) {};
                    \node[vertex] at (sts) {};
                    \node[vertex] at (tst) {};
                    \node[vertex] at (stst) {};
                    \node[vertex] at (tsts) {};

                    \node[below] at (e) {$e$};
                    \node[above] at (s) {$s$};
                    \node[above] at (t) {$t$};
                    \node[above] at (st) {$st$};
                    \node[above] at (ts) {$ts$};
                    \node[above] at (sts) {$sts$};
                    \node[above] at (tst) {$tst$};
                    \node[above] at (stst) {$stst$};
                    \node[above] at (tsts) {$tsts$};

                    \draw[arr,tarr] (e) -- (t);
                    \draw[arr,tstarr] (e.north) to [out=30,in=150] (tst.north);

                    \draw[arr,sarr] (s) -- (e);
                    \draw[arr,tarr] (s.south) to [out=330,in=210] (st.south);
                    \draw[arr,tstarr] (s.north) to [out=30,in=150] (stst.north);

                    \draw[arr,tstarr] (t) -- (st);

                    \draw[arr,sarr] (ts.south) to [out=330,in=210] (t.south);
                    \draw[arr,tarr] (ts.south) to [out=330,in=210] (tst.south);
                    \draw[arr,stsarr] (ts) -- (s);

                    \draw[arr,sarr] (sts.south) to [out=330,in=210] (st.south);
                    \draw[arr,tarr] (sts.south) to [out=330,in=210] (stst.south);
                    \draw[arr,stsarr] (sts.north) to [out=30,in=150] (e.north);

                    \draw[arr,sarr] (tsts.south) to [out=330,in=210] (tst.south);
                    \draw[arr,stsarr] (tsts.north) to [out=30,in=150] (t.north);

                    \draw[arr,darr] (st) -- (tst);
                    \draw[arr,darr] (tst) -- (stst);
                    \draw[arr,darr] (sts) -- (ts);
                    \draw[arr,darr] (tsts) -- (sts);

        \end{tikzpicture}
    \end{center}
\end{example}

We next compile a list of lemmas and corollaries which are useful for twisted length functions.
We start with a comparison between the twisted length between two elements and the identity and an element.

\begin{lemma}
    \label{lem:shorten_len}
    For $A \subseteq T$ and $v,w \in W$, then $\ell_A(v,w) = \ell_{v \cdot A}(wv^{\inv})$.
\end{lemma}
\begin{proof}
    \[
        \ell_{v \cdot A}(wv^{\inv}) = \ell(wv^{\inv}) - 2 \order{N(vw^{\inv}) \cap v\cdot A} = \ell_A(v,w).
    \]
\end{proof}

Just like the standard length function, the twisted length function is additive.
Therefore, knowing the twisted length according to $A$ from $u$ to $v$ and from $v$ to $w$, gives the twisted length from $u$ to $w$.

\begin{lemma}\cite[Prop 1.1]{Dyer_HeckeAlgebrasAndShellingsOfBruhatIntervalsIITWistedBruhatOrders}
    \label{lem:length_adds}
    For $A \subseteq T$ and $u,v,w \in W$, $\ell_A(u,v) + \ell_A(v,w) = \ell_A(u,w)$ and $\ell_A(u) + \ell_{u \cdot A}(v) = \ell_A(vu)$.
\end{lemma}

Since it is possible to add twisted lengths together, the twisted length from $v$ to $w$ are decomposable as the twisted length from the identity to $w$ with the twisted length from the identity to $v$ subtracted from it.
\begin{corollary}
    \label{cor:split_len}
    For $A \subseteq T$ and $v,w \in W$ we have $\ell_A(v,w) = \ell_A(w) - \ell_A(v)$.
\end{corollary}
\begin{proof}
    By \autoref{lem:shorten_len} $\ell_A(v,w) = \ell_{v\cdot A}(wv^{\inv})$ which by \autoref{lem:length_adds} gives us the desired results.
\end{proof}

Another expected result is that the twisted length is anti-commutative, \ie the twisted length from $u$ to $v$ is the negative twisted length from $v$ to $u$.
\begin{corollary}
    \label{cor:minus_length}
    For $A \subseteq T$, $u,v \in W$ then $\ell_A(u,v) = -\ell_A(v,u)$.
\end{corollary}
\begin{proof}
    Letting $w = u$ in \autoref{lem:length_adds} and noting $\ell_A(u,u) = 0$ gives us the desired result.
\end{proof}

This corollary boils down to a nice property on reflections in $T$ since every reflection is its own inverse.
\begin{corollary}
    \label{cor:refl_len}
    For $A \subseteq T$ and $t \in T$, $\ell_{t\cdot A}(t) = -\ell_A(t)$.
\end{corollary}
\begin{proof}
    By \autoref{lem:shorten_len} $\ell_{t \cdot A}(t) = \ell_A(t,e)$ which by \autoref{cor:minus_length} gives us the desired result.
\end{proof}

We next focus our attention to initial sections to get more information on reflections through the twisted length function.
Taking an element $w\in W$ and a reflection $t\in T$ then if the twisted length for an initial section $A$ increases from $w$ to $tw$ then the root corresponding to $t$ is not contained in the twisted inversion set $\twistedInv{A}{w}$.
\begin{lemma}\cite[Prop 1.2]{Dyer_HeckeAlgebrasAndShellingsOfBruhatIntervalsIITWistedBruhatOrders}
    \label{lem:length_vs_action}
    Let $A \in \initSects$, $w \in W$ and $t \in T$. Then $\ell_A(w,tw) > 0$ if and only if $t \notin w\cdot A$.
    Furthermore, $t \in T \bs A$ implies that $\ell_A(t) > 0$.
\end{lemma}

Since working with a length function is generally easier we give a condition on an edge $(u,v)$ being in the twisted Bruhat graph using the twisted length function.
In fact, it is exactly the condition we would expect to have: being an edge in the twisted Bruhat graph implies that the twisted length must be increasing.
This replicates the same behavior found in the standard Bruhat graph.
\begin{lemma}
    \label{lem:edge_implies_length_increase}
    For $A \in \initSects$ and  $u, v \in W$ such that $u = tv$ for some $t \in T$, then $(u,v) \in E_{(W,A)}$ implies $\ell_A(u,v) > 0$.
\end{lemma}
\begin{proof}
    For $(u,v) \in E_{(W,A)}$ by definition (see \autoref{eq:twisted-graph}) there exists a $t \in T$ such that $u = tv$ and $t \in v \cdot A$.
    By \autoref{lem:a_notin_wA_iff_a_in_twA} $t \notin tv\cdot A$, or in other words $t \in T \bs \left(tv\cdot A\right) = T \bs \left(u \cdot A\right)$.
    By \autoref{lem:length_vs_action} $\ell_{u \cdot A}(t) > 0$.
    Then by \autoref{lem:shorten_len} $\ell_A(u, v) > 0$ as desired.
\end{proof}

As in the standard Bruhat graph, another interpretation of this is that if $u$ is ``below'' $v$ in the twisted Bruhat order then the twisted length of $u$ must be less than the twisted length of $v$.
\begin{corollary}
    \label{cor:order_follows_length}
    For $u, v \in W$ and $A \in \initSects$ if $u \leqA v$ then $\ell_A(u) \leq \ell_A(v)$.
\end{corollary}
\begin{proof}
    Recall that by definition $u \leq_A v$ if and only if $C_n^A(u,v) \neq \varnothing$ for some $n \in \mathbb{N}$.
    In other words, there exists $(w_0, w_1, \ldots, w_n)$ such that $w_0 = u$, $w_n = v$ and $(w_i, w_{i+1}) \in E_{(W,A)}$.
    If $n = 0$, then $u = v$ and $\ell_A(u) = \ell_A(v)$ as desired.
    Else recall that $(w_i, w_{i+1}) \in E_{(W,A)}$ implies that $\ell_A(w_i, w_{i+1}) > 0$ by \autoref{lem:edge_implies_length_increase}.
    But by \autoref{lem:length_adds} we have
    \[
        \ell_A(u,v) = \ell_A(w_0,w_1) + \ell_A(w_1,w_2) + \cdots + \ell_A(w_{n-1},w_n)
    \]
    and since each component is positive we have $\ell_A(u,v) > 0$.
    Finally, by \autoref{cor:split_len} this implies $\ell_A(v) - \ell_A(u) > 0$ giving us the corollary.
\end{proof}

\subsection{Initial sections and Bruhat preclosure}
We finally put together the previous results to show that for an initial section $A \in \initSects$ then $\bclosure{A} = A$.
For this we will show many equivalence lemmas that, in the end, will give us the desired result.
We first begin by showing what it means for a reflection to be in $A$ with regard to twisted length.
Note that taking a root in an initial section will contribute to decreasing the twisted length.
This implies that the twisted length of this reflection must be below the identity, whose length never changes.
This is a direct result of \autoref{lem:length_vs_action} with $w = e$.
\begin{lemma}
    \label{lem:root_imp_len_lt_zero}
    For $t \in T$ and $A \in \initSects$ then $t \in A$ if and only if $\ell_A(t) < 0$.
\end{lemma}

In fact this lemma gives the reverse direction of \autoref{cor:order_follows_length} in the case of reflections when compared with the identity element.
In particular if a reflection has negative twisted length then it must be below the identity element in the twisted Bruhat order.
This gives us a quick way to determine if reflections are below or above the identity in the twisted Bruhat order.
\begin{lemma}
    \label{lem:length_iff_bruhat}
    For $t \in T$ and $A \in \initSects$, $\ell_A(t) < 0$ if and only if $t \leqA e$.
\end{lemma}
\begin{proof}
    If $t \leqA e$ then by \autoref{cor:order_follows_length}, $\ell_A(t) < \ell_A(e) = 0$ as desired.
    Suppose next that $\ell_A(t) < 0$.
    By \autoref{lem:root_imp_len_lt_zero} $\ell_A(t) < 0$ if and only if $t \in A$.
    Since $t \in e \cdot A$ we have $(t,e) \in E_{(W,A)}$.
    In particular, this implies $C_1^A(t,e) \neq \varnothing$ and $t \leqA e$ as desired.
\end{proof}

For the final lemma we show the only reflections that can be reached through $A$-paths starting from the identity are the reflections below the identity.
\begin{lemma}
    \label{lem:closure_twisted}
    For $t \in T$ and $A \in \initSects$ we have $t \leqA e$ if and only if $t \in \bclosure{A}$.
\end{lemma}
\begin{proof}
    Note that if $t \leqA e$ then by the previous two lemmas $t \in A$ and therefore $t \in \bclosure{A}$.
    Now suppose $t \in \bclosure{A}$.
    Then there exists an $A$-path from $e$ to $t$ in the Bruhat graph labelled by edges in $A$:
    \[
        e \xrightarrow{t_1} t_1 \xrightarrow{t_2} t_1 t_2 \xrightarrow{t_3} \ldots \xrightarrow{t_n} t_1 t_2 \ldots t_n = t.
    \]
    In other words $(t_1t_2 \ldots t_{i-1}, t_1 t_2 \ldots t_i) \in E$, and $\ell(t_1 t_2 \ldots t_{i-1}) < \ell(t_1 t_2 \ldots t_i)$ for all $i \in [n]$.
    Furthermore, since $t_i \in A$, this implies $(t_1 t_2 \ldots t_i, t_1 t_2 \ldots t_{i-1}) \in E_{(W,A)}$.
    In other words, $C_n^A(t,e) \neq \varnothing$.
    Therefore, $t \leqA e$.
\end{proof}

Putting together all the previous lemmas gives us the desired result.
\begin{theorem}
    \label{thm:Bruhat_is_closure}
    For $t \in T$ and $A \in \initSects$, $t \in \bclosure{A}$ if and only if $t \in A$.
\end{theorem}

\begin{remark}
    Note that this implies the Bruhat preclosure is nothing more than the identity operator on initial sections.
    Also, since inversion sets are finite initial sections, then for all $w \in W$ we have $N(w) = \bclosure{N(w)}$ for all Coxeter groups, not just finite ones.
\end{remark}
Additionally, we note that the above theorem was (independently) already known to Matthew Dyer.
Through personal communication \cite{Dyer_Communication}, Dyer mentioned that he provided a sketch of his proof at the ``Workshop on Bruhat order: recent developments and open problems''.
We were not at this workshop and therefore do not know of the specific techniques used in Dyer's proof sketch, but
through private communication with Dyer \cite{Dyer_Communication}, the above proof uses different techniques than those he used during the workshop and is the first known published proof of these results.

\section{Preclosure extensions}
\label{sec:preclosure-extensions}
In this section we discuss how to turn our preclosure into a closure using a particular extension.
It will turn out that this extension gives us a weaker version of the results we desire.
Although the definition of finite type is well-known, we could not find proofs of the following two lemmas, although we assume it is already in the literature.
We thus provide proofs for completeness.

Let $Z$ be a set and $\cl: \mcP(Z) \to \mcP(Z)$ a preclosure operator.
We say that the preclosure $\cl$ is \definition{of finite type} if $\cl(X) = \bigcup\set{\cl(Y)}{Y \subseteq X \text{ and } \order{Y} \in \N}$ for $X \subseteq Z$.
Let $\cl^{i}(X) = \cl(\cl(\cl(\cdots \cl(X))))$ where $\cl$ is performed $i$ times.
Furthermore, let $\cl^{\infty}(X) = \cup_{i \in \N} \cl^{i}(X)$.
It turns out that for a preclosure $\cl$ of finite type, then $\cl^{\infty}$ is always a closure operator.
To show this, we first give the following lemma:
\begin{lemma}
    \label{lem:closure_is_comm}
    Let $\cl: \mcP(Z) \to \mcP(Z)$ be a preclosure operator of finite type for a set $Z$.
    Then we have $\cl^{\infty} \circ \cl = \cl \circ \cl^{\infty} = \cl^{\infty}$.
\end{lemma}
\begin{proof}
    By construction $\cl(\cl^{\infty}(X)) = \cl(\cup_{i \in \N} \cl^{i}(X))$.
    But since $\cl$ is of finite type we have
    \[
        \cl\left(\bigcup_{i \in \N} \cl^{i}(X)\right) = \bigcup_{\substack{|Y|< \infty,\\ Y \subseteq \bigcup_{i \in \N} \cl^{i}(X)}} \!\!\!\!\!\!\!\!\!\cl(Y).
    \]
    Since $Y$ is finite, there is some $j$ for each $Y$ such that $Y \subseteq \cl^{j}(X)$ since $\cl^{i}(X) \subseteq \cl^{i+1}(X)$ by preclosure.
    Therefore,
    \[
        \bigcup_{\substack{|Y|< \infty,\\ Y \subseteq \bigcup_{i \in \N} \cl^{i}(X)}} \!\!\!\!\!\!\!\!\!\cl(Y) = \bigcup_{j \in \N}\left( \bigcup_{Y \subseteq \cl^{j}(X)} \!\!\!\cl(Y)\right) = \bigcup_{j \in \N} \cl^{j+1}(X) = \cl^{\infty}(X).
    \]

    Finally, 
    \[
       \cl^{\infty}(\cl(X)) = \bigcup_{i \in \N} \cl^{i}\left( \cl(X)\right) = \bigcup_{i \in \N}\cl^{i + 1}(X) = \cl^{\infty}(X).
    \]
\end{proof}

\begin{lemma}
    \label{lem:finite_type_implies_infty_closure}
    Let $\cl: \mcP(Z) \to \mcP(Z)$ be a preclosure operator of finite type for a set $Z$.
    Then $\cl^{\infty}$ is a closure operator.
\end{lemma}
\begin{proof}
    We first remark that $\cl^{\infty}$ is necessarily a preclosure operator by construction.
    Therefore, it suffices to show idempotence: $\cl^{\infty}(\cl^{\infty}(X)) = \cl^{\infty}(X)$.
    But by, \autoref{lem:closure_is_comm} we have our desired result.
\end{proof}

We next show that the Bruhat preclosure is of finite type.
For this we let $\bnclosure{A}{i}$ denote the Bruhat preclosure performed $i$ times, \ie $\cnclosure{A}{1} = \cclosure{A}$, $\cnclosure{A}{2} = \cclosure{\cclosure{A}}$, $\cnclosure{A}{3} = \cclosure{\cclosure{\cclosure{A}}}$, etc., and we let $\cnclosure{A}{\infty} = \cup_{i \in \N} \cnclosure{A}{i}$.
\begin{proposition}
    For $A \subseteq T$ the Bruhat preclosure is of finite type.
\end{proposition}
\begin{proof}
    To show the Bruhat preclosure is of finite type, it suffices to show 
    \[
        \bclosure{A} = \bigcup_{\substack{B\text{ finite}\\ B \subseteq A}} \bclosure{B}.
    \]
    By definition of preclosure, $\bclosure{A} \supseteq \bigcup_{B\text{ finite}, B\subseteq A} \bclosure{B}$.
    It remains to show the other inclusion.

    By definition of the Bruhat preclosure for $t \in \bclosure{A}$ there exists a finite path
    \[
        e \xrightarrow{t_1} t_1 \xrightarrow{t_2} t_1t_2 \xrightarrow{t_3} \ldots \xrightarrow{t_n} t_1 t_2 \ldots t_n = t.
    \]
    Therefore, taking $B = \left\{ t_1, \ldots, t_n \right\} \subseteq A$ implies $t \in \bigcup_{B\text{ finite}, B\subseteq A} \bclosure{B}$ as desired.
\end{proof}

\begin{corollary}
    \label{cor:inf_is_closre}
    For $A \subseteq T$, $\bbclosure{A}$ is a closure operator.
\end{corollary}
We call the closure $\bbclosure{A}$ the \definition{infinite Bruhat closure}.
Note that the assumption of finite type is necessary.
There exist preclosures such that $\cl^{\infty}$ is not a closure without the assumption of finite type.
An example of this can be seen next.
\begin{example}
    We construct a preclosure operator by combining two operators.
    Let $X \subseteq \Q$ and define the following operators:
    \[
        d(X) = \set{\frac{a}{b}, \frac{a}{b+1}}{\frac{a}{b} \in X},\quad i(X) = \begin{cases*}
            X \cup \inf(X) & if $\inf(X)$ exists\\
            X               & otherwise
        \end{cases*}
    \]
    and set $\cl{X} = i(d(X))$.
    Taking $X = \left\{ 1 \right\}$ we can see that $\cl^{\infty}(X) = \left\{\frac{1}{1}, \frac{1}{2}, \frac{1}{3}, \ldots \right\}$, but $\cl(\cl^{\infty}(X)) = \cl^{\infty}(X) \cup \left\{ 0 \right\}$.
    Therefore, $\cl^{\infty}$ is not a closure operator.
\end{example}

\section{Main Results}
\label{sec:main-results}
In this section we prove a weaker version of Dyer's conjecture for an arbitrary Coxeter system using the infinite Bruhat closure.
To do this we first give some lemmas on finite Coxeter systems and then prove Dyer's conjecture on dihedral Coxeter groups.
We then use these in order to prove our main theorem.
\begin{theorem}\label{thm:main}
    For $u$ and $v$ in $W$ bounded above, then
    \[
        N(u \joinR v) = \bbclosure{N(u) \cup N(v)}.
    \]
\end{theorem}

\subsection{Initial lemmas on finite Coxeter systems}\label{ssec:initial-lemmas-on-finite-coxeter-systems}
In this section we prove that if the join of two elements in a finite Coxeter group is the long element, then Dyer's conjecture, and in particular \autoref{thm:main}, is true.

Let $(W,S)$ be a finite Coxeter system.
Recall that for a $w$ in $W$ the set of inversions is the set ${N(w) = \set{t \in T}{\ell(tw) < \ell(w)}}$.
The set of left descents is the set $D_L(w) = N(w) \cap S =  \set{s \in S}{\ell(sw) < \ell(w)}$.

We require the following corollary which is a restriction of \cite[Corollary 1.6(ii)]{Dyer_OnTheWeakOrderOfCoxeterGroups}.
\begin{corollary}\label{cor:simples_contained}
    If $X \subseteq W$ is bounded above in the weak order and $s \in S$ then $s \in N(\bigjoin X)$ if and only if $s \in N(x)$ for some $x \in X$.
\end{corollary}


We next give three additional lemmas before proving the result for this section: that if the join of two elements is equal to the long element then the Bruhat preclosure of their inversion sets is equal to the set of reflections, \ie $u,v \in W$ such that $u \joinR v = \longEl$ implies $N(\longEl) = \bclosure{N(u) \cup N(v)}$.
To do this we go through the descents of each element.
We will show that if every simple reflection is contained in one of the two descent sets then the join must be the long element.
For this we need the following lemma.

\begin{lemma}[{\cite[Prop 2.3.1 (ii)]{Bjorner_CombinatoricsofCoxeterGroups}}]
    \label{lem:all_descent_implies_long}
    Let $(W,S)$ be finite.
    For $w \in W$, if $D_L(w) = S$ then $w = \longEl$.
\end{lemma}

The following is well-known, but we provide a proof for completeness.
\begin{lemma}\label{lem:S_iff_long}
    Let $(W,S)$ be a finite Coxeter system and $u,v \in W$.
Then $u \joinR v = \longEl$ if and only if $D_L(u) \cup D_L(v) = S$.
\end{lemma}
\begin{proof}
    If $D_L(u) \cup D_L(v) = S$ then $S = D_L(u) \cup D_L(v) \subseteq D_L(u \joinR v) \subseteq S$ and by \autoref{lem:all_descent_implies_long} we have $u \joinR v = \longEl$.

    Conversely, if $u \joinR v = \longEl$ then $D_L(u \joinR v) = S$ and by definition $S \subseteq N(u \joinR v)$.
    By \autoref{cor:simples_contained} for every $s \in S$, $s \in N(u) \cup N(v)$ and therefore $s \in D_L(u) \cup D_L(v)$ giving us the desired result.
\end{proof}

When it comes to the Bruhat preclosure there is a nice relationship between the simple reflections and the reflections.
Recall that the Bruhat graph restricted to $S$ is just the Cayley graph of the weak order.
This implies that the Bruhat preclosure of $S$ will be the reflections reachable in the weak order.
Since the Cayley graph of the right weak order has all of $W$ as a vertex set, then every reflection can be reached through this graph.
This is true for any Coxeter system, not just finite ones, giving the following lemma.

\begin{lemma}
    \label{lem:closure_simple_roots}
    Let $(W, S)$ be an (arbitrary) Coxeter system, then
    $\bclosure{S} = T$.
\end{lemma}

For the long element, since every simple root is in the inversion set, we can trace the Cayley graph of the weak order to obtain a path to every reflection as we did in the previous lemma.
We therefore have the following proposition.
\begin{proposition}
    \label{prop:long_implies_closed}
    Let $\longEl$ be the long element of a finite Coxeter system $(W,S)$.
    For $u,v \in W$, if we have $u \joinR v = \longEl$ then $N(u \joinR v) = \bclosure{N(u) \cup N(v)}$
\end{proposition}
\begin{proof}
    Since $u \joinR v = \longEl$, by \autoref{lem:S_iff_long} we have $D_L(u) \cup D_L(v) = S$ and by definition $S \subseteq N(u) \cup N(v)$.
    Then by \autoref{lem:closure_simple_roots} and \autoref{thm:Bruhat_is_closure}
    \[
        T = \bclosure{S} \subseteq \bclosure{N(u) \cup N(v)} \subseteq \bclosure{N(u \join_R v)} = N(u \join_R v) = T
    \]
    as desired.
\end{proof}

\subsection{Dihedral Coxeter groups}
\label{subsec:dihedral}

The join in finite dihedral Coxeter groups is well-behaved: the join of two elements is either one of the elements itself or is the long element.
In the infinite case, we have a similar phenomenon, but additionally the join between two elements might not exist.
This allows us to give a proof of our main theorem in the dihedral case which will be needed later.
Note that in \cite[Theorem 3.1]{Biagioli2025} Biagioli and Perrone give a proof of this result.
We give an alternative proof using our previous lemmas.
\begin{theorem}
    \label{thm:join_closure_dihedral}
    Let $(W,S)$ be a dihedral Coxeter system.
    For $u, v \in W$, if $u \joinR v$ exists then $N(u \joinR v) = \bclosure{N(u) \cup N(v)}$.
\end{theorem}
\begin{proof}
    Let $S = \left\{ s,t \right\}$ and without loss of generality suppose $\ell(u) \geq \ell(v)$.
    If $D_L(u) \cup D_L(v) = S$ then $u \joinR v$ exists only if $W$ is finite in which case $u \joinR v = \longEl$ by \autoref{lem:S_iff_long}.
    Then \autoref{prop:long_implies_closed} gives us the desired result.
    If $D_L(u) \cup D_L(v) = \varnothing$ then $u = v = e$ and $N(e \joinR e) = \varnothing = \bclosure{N(e) \cup N(e)}$.
    
    Finally, we suppose without loss of generality that $D_L(u) \cup D_L(v) = \left\{ s \right\} \subsetneq S$.
    Since $W$ is dihedral there are reduced expressions for $u$ and $v$, $u = ststs\ldots$ and $v = ststs\ldots$ of lengths $\ell(u)$ and $\ell(v)$ respectively.
    In other words $v \leqR u$ implying $N(v) \subseteq N(u)$ and $u \joinR v = u$.
    By \autoref{thm:Bruhat_is_closure}
    \[
        N(u \joinR v) = N(u) = \bclosure{N(u)} = \bclosure{N(u) \cup N(v)}
    \]
    as desired.
\end{proof}

\subsection{General Case}
We now consider the general case of an arbitrary Coxeter group.
We start off with some preliminary lemmas which will help build up to our final theorem.
The following lemma is an easy exercise using the properties of preclosure operators.
\begin{lemma}
    \label{lem:subset_inclusion_closure}
    For two subsets $A$ and $B$ of $T$ we have
    \[
        \bclosure{A} \cup B \subseteq \bclosure{A \cup B}.
    \]
\end{lemma}

Since the join of the weak order exists only when we have sets of elements bounded above we will exclusively work with elements having bounds above.
If $u$ and $v$ are bounded above by an element $w$ then the Bruhat closure of their inversion sets must be contained in the inversion set of $w$.
The main idea comes from the fact that since the Bruhat closure of an inversion set is the inversion set itself, we can replace the smaller inversion sets for the larger one to get our desired result.
\begin{lemma}
    \label{lem:upper_bound_closure}
    Let $w \in W$ be an upper bound of $u$ and $v$ in $W$.
    Then $\bclosure{N(u) \cup N(v)} \subseteq N(w)$.
\end{lemma}
\begin{proof}
    By construction $u \leqR w$ and $v \leqR w$.
    In other words $N(u) \subseteq N(w)$ and $N(v) \subseteq N(w)$, and in particular $N(u) \cup N(v) \subseteq N(w)$.
    Therefore,
    \[
        \bclosure{N(u) \cup N(v)} \subseteq \bclosure{N(w)} = N(w)
    \]
    where the last equality comes from \autoref{thm:Bruhat_is_closure}.
\end{proof}

Given a subset $I$ of $S$ recall that $W^{I}$ is the set of minimal coset representatives of $W_I$, \ie
\[
    W^{I} = \set{w \in W}{\ell(w) < \ell(ws) \text{ for all } s \in I}.
\]
An interesting thing to note is that for all reduced expressions of minimal coset representatives $x \in W^{I}$ then multiplying on the right by a reduced expression of an element $w$ in $W_I$ gives a reduced expression for $xw$.
This falls from the fact that the poset structure between $x$ and $x\longElI{I}$ is the same as the poset structure between the identity and $\longElI{I}$.
We use this idea in the following.

\begin{lemma}
    \label{lem:input_x}
    Let $(W,S)$ be a Coxeter system and let $I$ be a subset of $S$.
    Then for $x \in W^{I}$ we have
    \[
        x\left(\bclosure{I}\right)x^{\inv} = \bclosure{xIx^{\inv}}
    \]
\end{lemma}
\begin{proof}
    Let $(W_I, I)$ be the Coxeter system associated to $I$.
    Given an $I$-path 
    \[
        P : e \xrightarrow{t_1} t_1 \xrightarrow{t_2} t_1t_2 \xrightarrow{t_3} \cdots \xrightarrow{t_n} t_1t_2 \ldots t_n
    \]
    we aim to show that $P$ exists if and only if there exists a $\left(xIx^{\inv}\right)$-path
    \[
        P' : e \xrightarrow{xt_1x^{\inv}} xt_1x^{\inv} \xrightarrow{xt_2x^{\inv}} xt_1t_2x^{\inv} \xrightarrow{xt_3x^{\inv}} \cdots \xrightarrow{xt_nx^{\inv}} xt_1t_2\ldots t_nx^{\inv}.
    \]
    We do this by showing for a given $i \in [n]$, then we have
        $t_1 t_2 \ldots t_{i-1} \xrightarrow{t_i} t_1 t_2 \ldots t_{i-1} t_i$
        if and only if $xt_1 t_2 \ldots t_{i-1}x^{\inv} \xrightarrow{xt_ix^{\inv}} xt_1 t_2 \ldots t_{i-1}t_i x^{\inv}$.

    For ease of notation, let $w = t_1 \ldots t_{i-1}$.
    Recall that $xwx^{\inv} \xrightarrow{xt_ix^{\inv}} x w t_ix^{\inv} = (xwt_iw^{\inv} x^{\inv})(xwx^{\inv})$ if and only if $xw t_i w^{\inv} x^{\inv} \notin N(x w x^{\inv})$.
    Recalling that inversion sets are reflection cocycles we have
    \[
        x w t_i w^{\inv} x^{\inv} \notin \left( N(x) \symdiff xN(w)x^{\inv} \right) \symdiff x w N(x^{\inv}) w^{\inv} x^{\inv}.
    \]
    In other words, either $x w t_i w^{\inv} x^{\inv}$ is contained in exactly zero or two of the sets in the symmetric difference.

    By construction, since $x \in W^{I}$ then $x w t_i w^{\inv} x^{\inv} \notin N(x)$.
    Suppose contrarily that $xw t_i w^{\inv} x^{\inv} \in N(x)$ and note that $w t_i w^{\inv} \in W_I$.
    This implies $\ell(xw t_i w^{\inv}) = \ell(xw t_i w^{\inv} x^{\inv} x) < \ell(x)$ which is a contradiction to the fact that $x \in W^{I}$.

    Additionally, as $x \in W^{I}$ then $x w t_i w^{\inv} x^{\inv} \notin xwN(x^{\inv})w^{\inv} x^{\inv}$.
    Indeed, since $x \in W^{I}$ then $\ell(x^{\inv}) = \ell(x) < \ell(xt_i) = \ell(t_ix^{\inv})$ giving us $t_i \notin N(x^{\inv})$.
    
    Since $x w t_i w^{\inv} x^{\inv}$ is not contained in two of the sets, this implies it is contained in none of the sets and in particular $x w t_i w^{\inv} x^{\inv} \notin xN(w)x^{\inv}$.
    In other words $x w t_i w^{\inv} x^{\inv} \notin N(x w x^{\inv})$ if and only if $x w t_i w^{\inv} x^{\inv} \notin xN(w)x^{\inv}$.
    But this latter is true if and only if $w t_i w^{\inv} \notin N(w)$ or in other words if and only if there is an edge $w = t_1 t_2 \ldots t_{i-1} \xrightarrow{t_i} t_1 t_2 \ldots t_{i-1} t_i$.
\end{proof}

Letting $I$ contain two simple reflections gives us the dihedral Coxeter groups.
From \autoref{subsec:dihedral} we saw that for dihedral Coxeter groups we have a stronger result.
We can use this and our previous lemma to ``move'' this idea into an arbitrary part of our Coxeter group for $x \in W^I$.
\begin{lemma}
    \label{lem:dihedral_case}
    Let $(W, S)$ be a Coxeter system, $I = \left\{ s,t \right\} \subseteq S$ with $W_I$ finite and $x \in W^{I}$ be a minimal coset representative.
    Then
    \[
        N(x\longElI{I}) = \bclosure{N(xs) \cup N(xt)}.
    \]
\end{lemma}
\begin{proof}
    First notice that $x\longElI{I} = xs \joinR xt$ by construction.
    By \autoref{lem:upper_bound_closure} we have $\bclosure{N(xs) \cup N(xt)} \subseteq N(x\longElI{I})$.
    
    For the other direction, notice that
    \begin{align*}
        N(x\longElI{I}) &= N(x) \sqcup xN(\longElI{I}) x^{\inv}\\
        &= N(x) \sqcup x\bclosure{I} x^{\inv}\\
        &= N(x) \sqcup \bclosure{xIx^{\inv}}\\
        &\subseteq \bclosure{N(x) \sqcup xIx^{\inv}}\\
        &\subseteq \bclosure{N(xs) \cup N(xt)}
    \end{align*}
    Where the first equality is due to the reflection cocycle property and the fact that $N(x) \cap xN(\longElI{I}) x^{\inv} = \emptyset$, second equality by \autoref{prop:long_implies_closed}, the third equality by \autoref{lem:input_x}, the first subset by \autoref{lem:subset_inclusion_closure}, and the second subset by the reflection cocycle property.
\end{proof}
The idea of the above lemma is that if we have two elements $u$ and $v$ which cover another element $x$ then Dyer's conjecture is true.
This can be weakened so that $u$ and $v$ don't necessarily need to cover $x$, but instead can be weakly above $x$ (so that $x = u \meetR v$).
In the next lemma we keep $x \coverR v$, but we allow $u$ to be weakly above $x$.
Recall that $\bbclosure{A} = \cup_{i \in \N} \bnclosure{A}{i}$ where $\bnclosure{A}{i}$ is shorthand notation for applying the Bruhat preclosure $i$ times on $A$.
The following figure will be useful in the next lemma and the theorem.
\begin{center}
    \begin{figure}[H]
        \begin{tikzpicture}
            [vertex/.style={inner sep=1pt,draw=black, circle, fill=black, thick}]
            \coordinate (x) at (0,0);
            \coordinate (u) at (-1,1);
            \coordinate (xs) at (1,1);
            \coordinate (j) at (0,2);
            \coordinate (v) at (2,2);
            \coordinate (w) at (1,3);

            \draw (x) -- (u);
            \draw (x) -- (xs);
            \draw (xs) -- (v);
            \draw (v) -- (w);
            \draw (u) -- (j);
            \draw (j) -- (w);
            \draw[dotted] (xs) -- (j);

            \node[vertex] at (x) {};
            \node[vertex] at (u) {};
            \node[vertex] at (xs) {};
            \node[vertex] at (j) {};
            \node[vertex] at (v) {};
            \node[vertex] at (w) {};

            \node[below] at (x) {$x = u \meetR v$};
            \node[left] at (u) {$u$};
            \node[right] at (xs) {$xs$};
            \node[left] at (j) {$u \joinR xs$};
            \node[right] at (v) {$v$};
            \node[above] at (w) {$w$};
        \end{tikzpicture}
        \label{fig:lemmas}
        \caption{Poset structure when $u$ and $v$ are bounded above by $w$.}
    \end{figure}
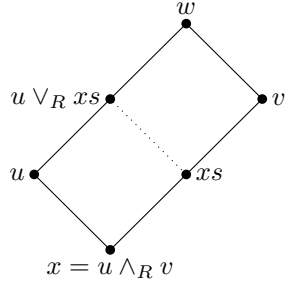
\end{center}

\begin{lemma}\label{lem:proof-one-cover}
    Let $(W, S)$ be a Coxeter system.
    For $u$ and $v$ in $W$ bounded above by an element $w \in W$ such that $x = u \meetR v$, $v = xs$ for some $s \in S$, $\ell(u) \geq \ell(v)$ and $v \not\leqR u$.
    Then there is a $t \in S \bs \{s\}$ such that
    \[
        \bbclosure{N(u) \cup N(v)} = \bbclosure{N(u) \cup N(x\longElI{\left\{s, t\right\}})}
    \]
    and
    \[
        u \joinR v = u \joinR x\longElI{\left\{s, t\right\}}.
    \] 
\end{lemma}
\begin{proof}
    By construction, since $v = xs \not\leqR u$ and $\ell(u) \geq \ell(v)$ then there exists a $t \in S \bs \{s\}$ such that $x \coverR xt \leqR u$.
    Then by \autoref{lem:dihedral_case} we have $\bclosure{N(xs) \cup N(xt)} = N(x\longElI{\left\{s, t\right\}})$ implying
    \[
        N(u) \cup N(xs)  \subseteq N(u) \cup \bclosure{N(xt) \cup N(xs)} = N(u) \cup N(x\longElI{\left\{s, t\right\}})
    \]
    as $N(xt) \subseteq N(u)$.
    Furthermore,
    \[
        N(u) \cup N(x\longElI{\left\{s, t\right\}}) = N(u) \cup \bclosure{N(xt) \cup N(xs)} \subseteq \bclosure{N(u) \cup N(xs)}
    \]
    by \autoref{lem:subset_inclusion_closure}.
    Finally, \autoref{lem:closure_is_comm} and \autoref{cor:inf_is_closre},
    \[
        \bbclosure{N(u) \cup N(xs)} \subseteq \bbclosure{N(u) \cup N(x\longElI{\left\{s, t\right\}})} \subseteq \bbclosure{\bclosure{N(u) \cup N(xs)}} = \bbclosure{N(u) \cup N(xs)}
    \]
    and
    $u \joinR v = u \joinR xt \joinR xs = u \joinR x\longElI{\left\{s, t\right\}}$ as desired.
\end{proof}

Note that for finite Coxeter groups, since $T$ is finite, we have the following corollary.
\begin{corollary}
    \label{cor:proof-one-cover-finite}
    Let $(W, S)$ be a finite Coxeter system such that $x = u \meetR v$, $v = xs$ for some $s \in S$, $\ell(u) \geqR \ell(v)$ and $v \not\leqR u$.
    Then there is a $t \in S \bs \{s\}$ such that
    \[
        \bnclosure{N(u) \cup N(v)}{i} = \bnclosure{N(u) \cup N(x\longElI{\left\{s, t\right\}})}{i}
    \]
    where $i = \order{T} - \order{N(u) \cup N(v)}$.
\end{corollary}
\begin{proof}
    As $(W, S)$ is finite, then in the worst case scenario, we only add one reflection each time we take the Bruhat preclosure.
    As there are $\order{T} - \order{N(u) \cup N(v)}$ reflections we can add, this is the maximum number of closures needed.
\end{proof}

We now give one final technical lemma before proving our main theorem.
\begin{lemma}
    \label{lem:one-jump-in-algo}
    For $a, b \in W$ let $x = a \join_R b$.
    If there exists $I = \left\{s, t\right\} \subseteq S$ (with $s$ and $t$ distinct) such that $x \coverR xt \leqR a$, $x \coverR xs \leqR b$ and $b \leq_R x\longElI{I}$
    then:
    \[
        a \join_R b = a \join_R x\longElI{I} \text{ and } \bbclosure{N(a) \cup N(b)} = \bbclosure{N(a) \cup N(x\longElI{I})}.
    \]
\end{lemma}
\begin{proof}
    It's clear that
    \[
        a \joinR b = a \joinR xs \joinR xt \joinR b = a \joinR x\longElI{I} \joinR b = a \joinR x\longElI{I}
    \]
    since $x \coverR xt \leqR a$, $x \coverR xs \leqR b$, $b \leq_R x\longElI{I}$ and $x \in W^I$.
    For the infinite Bruhat closure we have:
    \begin{align*}
        \bbclosure{N(a) \cup N(b)} &= \bbclosure{N(a) \cup N(xs) \cup N(xt) \cup N(b)}\\
            &\subseteq \bbclosure{N(a) \cup N(x\longElI{I}) \cup N(b)}\\ 
            &= \bbclosure{N(a) \cup N(x\longElI{I})}\\
            &= \bbclosure{N(a) \cup \bclosure{N(xs) \cup N(xt)}}\\
            &= \bbclosure{N(a) \cup N(xs) \cup N(xt)}\\
            &\subseteq \bbclosure{N(a) \cup N(b)}.
    \end{align*}
    Thus completing our proof.
\end{proof}

We are now ready to prove our main theorem.
\begin{reptheorem}{thm:main}
    For $u$ and $v$ in $W$ such that there is a $w \in W$ with $u \leqR w$ and $v \leqR w$, then
    \[
        N(u \joinR v) = \bbclosure{N(u) \cup N(v)}.
    \]
\end{reptheorem}
\begin{proof}
    Without loss of generality, we suppose that $\ell(v) \leq \ell(u)$.
    Note that if $v \leqR u$ then by \autoref{thm:Bruhat_is_closure} we're done:
    \[
        N(u \joinR v) = N(u) = \bclosure{N(u)} = \bclosure{N(u) \cup N(v)}.
    \] 
    Therefore, we assume that $u$ and $v$ are incomparable.

    We will next describe an algorithm to produce our desired results.
    Set $\msE = \{v\}$ to be an ordered set and set $a = u$.
    We iterate over the set $\msE$ until it is empty.

    \textbf{Iteration:} Let $b$ be the final element in $\msE$ and let $x = a \meetR b$.
    There exists $I = \left\{s, t\right\} \subseteq S$ such that $x \coverR xt \leqR a$ and $x \coverR xs \leqR b$.
    We have four cases to consider depending on if $a, b \leqR x\longElI{I}$.
    \begin{enumerate}[leftmargin=16pt]
        \item If neither $a$ nor $b$ are below $x\longElI{I}$ in the weak order, then by \autoref{lem:proof-one-cover} we have $\bbclosure{N(a) \cup N(xs)} = \bbclosure{N(a) \cup N(x\longElI{I})}$ and $a \joinR xs = a \joinR x\longElI{I}$.
                Therefore, we add $x\longElI{I}$ into the set $\msE$ as the final element and iterate. Our new $b$ will be $x\longElI{I}$ in the next iteration as it will now be the final element in $\msE$, but $a$ stays the same.
        \item If $b \leqR x\longElI{I}$ but $a \not\leqR x\longElI{I}$, then by \autoref{lem:one-jump-in-algo} we have $\bbclosure{N(a) \cup N(b)} = \bbclosure{N(a) \cup N(x\longElI{I})}$ and $a \joinR b = a \joinR x\longElI{I}$.
                Therefore, we add $x\longElI{I}$ into the set $\msE$ as the final element, remove $b$ from $\msE$ and iterate. Our new $b$ will be $x\longElI{I}$ in the next iteration as it will now be the final element in $\msE$, but $a$ stays the same.
        \item If $a \leqR x\longElI{I}$ but $b \not\leqR x\longElI{I}$, then by \autoref{lem:one-jump-in-algo} we have $\bbclosure{N(a) \cup N(b)} = \bbclosure{N(x\longElI{I}) \cup N(b)}$ and $a \joinR b = x\longElI{I} \joinR b$.
                Therefore, we replace $a$ with $x\longElI{I}$ and iterate. Our $b$ will stay the same as $b$ remains the final element in $\msE$.
        \item If $a, b \leqR x\longElI{I}$, then by \autoref{lem:one-jump-in-algo} we have $\bbclosure{N(a) \cup N(b)} = \bbclosure{N(x\longElI{I})}$ and $a \joinR b = x\longElI{I}$.
                Therefore, we replace $a$ with $x\longElI{I}$ and remove the final element (the $b$) from $\msE$. Our new $b$ will be whatever the final element of $\msE$ is, assuming there are elements remaining.
    \end{enumerate}
    
    Since $u$ and $v$ are bounded and since the interval $[e, w]_R$ is finite, this algorithm will eventually stop in the final case where $a$ is some element and $\msE$ is empty.
    Then by our chains of equality of joins, we have $a = u \joinR v$ and our chain of Bruhat preclosures gives us:
    \[
        \bbclosure{N(u) \cup N(v)} = \bbclosure{N(a)} = \bbclosure{N(u \joinR v)} = N(u \joinR v)
    \] 
    as desired.
\end{proof}

Note that the biggest hurdle in proving Dyer's conjecture is~\autoref{lem:proof-one-cover}.
Various tests through SageMath shows that we should be able to replace the infinite Bruhat closure with the Bruhat preclosure and we see no inherent reason why the Bruhat preclosure is not sufficient.
We conjecture that it should be possible to remove the $\infty$ and only need the Bruhat preclosure for \autoref{lem:proof-one-cover}.
Then our proof of \autoref{thm:main} will naturally give us the desired result of Dyer's conjecture.
If the Bruhat preclosure was a closure then this theorem would imply that Dyer's conjecture is true.
We use this fact to give a second proof that Dyer's conjecture is true for type $A$.

Further, note that in the finite case, by \autoref{cor:proof-one-cover-finite}, we have the following corollary.
\begin{corollary}
    For $u$ and $v$ in a finite Coxeter group $W$, then there is $i \in \N$ such that
    \[
        N(u \joinR v) = \bnclosure{N(u) \cup N(v)}{i}
    \]
    where $i \leq |T| - |N(u) \cup N(v)|$.
\end{corollary}

\section{Type A}
\label{sec:type-a-case}

\newcommand{\perm}[2]{\left(\begin{smallmatrix}#1\\#2\end{smallmatrix}\right)}

In this section we show that in type $A$, the Bruhat preclosure is a closure, this giving a second proof of the conjecture of Dyer for type $A$ Coxeter groups: $N(u \joinR v) = \bclosure{N(u) \cup N(v)}$.
We recall that the proof of Dyer's conjecture for type $A$ was proved independently by Biagioli and Perrone in \cite[Theorem 4.4]{Biagioli2025}.
Furthermore, we note that Biagioli and Perrone use the left Bruhat preclosure whereas we use the right Bruhat preclosure (they multiply on the left rather than on the right).

Recall that a presentation for a type $A_n$ Coxeter group is given by the following:
\[
    W = \genset{S}{ (s_i s_{i+1})^3 = e \text{ for } i \in [n-1],\, (s_i s_j)^2 = e \text{ for } \order{j - i} > 1,\text{ and } s_i^2 = e}.
\]
We use the symmetric group $S_{n+1}$ representation of the type $A_n$ Coxeter group where the simple reflections are given by adjacent transpositions $s_i = (i, i + 1)$ and reflections are given by transpositions $t = (a, b) \in T$.
Elements of $W$ are then represented as permutations of the numbers $1$ through $n+1$ which we denote using both cycle notation and two-line notation.
For example, in $S_3$, the permutation $w \in S_3$ in cycle notation $w = (1, 2, 3)$ is associated to the two-line notation $w = \perm{1&2&3}{2&3&1}$.
By abuse of notation, for a permutation $w \in S_n$, for $i \in [n]$ we let $w(i)$ represent the number in the second row and $i$-th column of the two-line notation for $w$, \ie in our example we have $w(3) = 1$.

The inversion set of a permutation $w$ is a pair $(i, j)$ such that $i < j$ and $w(i) > w(j)$.
This coincides with the inversion set defined in \autoref{ssec:coxeter-groups}
There are two ways we can apply reflections on permutations: on their positions or on their values.
Multiplying from left to right, \ie right multiplication, applies the reflection to the values whereas multiplying from right to left, \ie left multiplication, applies the reflection to the positions.
We see these definitions in the following example.

\begin{example}
    Let $(W, \left\{s, t\right\})$ be a type $A_2$ Coxeter group where $s = (1, 2)$ and $t = (2, 3)$ are the associated simple transpositions in $S_3$.
    Take $w$ to be the element $w = st = (1, 2) (2, 3) = (1, 3, 2)$ where the cycle notation is multiplied from left to right.
    Then $st$ in two-line notation is given by $\perm{1&2&3}{3&1&2}$ as $1 \to 3$, $2 \to 1$ and $3 \to 2$.
    Calculating the inversion in the standard way gives us $N(st) = \left\{s, sts\right\}$.
    Using the permutation definition of inversion set we have that since $1 < 2$ and $st(1) = 3 > 1 = st(2)$ then $(1, 2) = s$ is an inversion.
    Similarly, since $1 < 3$ and $st(1) = 3 > 2 = st(3)$ then $(1, 3) = sts$ is an inversion.
    Notice that the two inversion sets coincide!

    We now apply the transpositions from both directions to understand what happens when we apply transpositions.
    \begin{align*}
        \text{Left multiplication: }\perm{1&2&3}{1&2&3} \xrightarrow{(2, 3)} \perm{1&2&3}{1&3&2} \xrightarrow{(1, 2)} \perm{1&2&3}{3&1&2}\\
        \text{Right multiplication: }\perm{1&2&3}{1&2&3} \xrightarrow{(1, 2)} \perm{1&2&3}{2&1&3} \xrightarrow{(2, 3)} \perm{1&2&3}{3&1&2}
    \end{align*}
    In the first line, we perform the multiplication from right to left, implying we swap positions.
    In the second line, we perform the multiplication from left to right, implying we swap values.
    In this second case, notice that each operation also adds the appropriate transposition into our inversion set.
    From $\perm{1&2&3}{1&2&3}$ to $\perm{1&2&3}{2&1&3}$ we swapped the values $1$ and $2$ in positions $1$ and $2$ adding $(1, 2)$ into our set of transpositions.
    From $\perm{1&2&3}{2&1&3}$ to $\perm{1&2&3}{3&1&2}$ we swapped the values $2$ and $3$ in positions $1$ and $3$ adding $(1, 3)$ into our set of transpositions and keeping $(1, 2)$ in our set.
\end{example}

We note that the proof technique used here is similar to the proof technique used in the proof of \cite[Theorem 4.4]{Biagioli2025} and was developed independently of the results of Biagioli and Perrone.
We note that although the proof technique is similar, our results are a generalization of, and thus strictly stronger than, the results by Biagioli and Perrone.
The main difference is that they show the following is true when $A = N(u) \cup N(v)$ with $u, v \in W$ whereas we let $A$ be any arbitrary subset of $T$.
\begin{lemma}
    \label{lem:A-reflection-only-between}
    Let $W$ be a type $A_n$ Coxeter group and let $A \subseteq T$.
    If $r = (a, b) \in \bclosure{A}$ (with $a < b$) then
    \[
        r = (a, c_1) (c_1, c_2) \dots (c_{m-2}, c_{m-1})(c_{m-1}, b)(c_{m-2}, c_{m-1})\dots (c_1, c_2)(a, c_1)
    \]
    where $a < c_1 < \ldots < c_{m-1} < b$ and $(c_i, c_{i+1}) \in A$ for all $i$, such that
    \[
       e \xrightarrow{(a, c_1)} (a, c_1)
       \xrightarrow{(c_1, c_2)} (a, c_1)(c_1, c_2)
       \xrightarrow{(c_2, c_3)}
       \ldots
       \xrightarrow{(1, c_1)} r
    \] 
    is an $A$-path.
\end{lemma}
\begin{proof}
    First, notice that since $r \in \bclosure{A}$ then $r = t_1 t_2 \ldots t_\ell$ where each $t_i \in A$ and each multiplication on the right, increases the length by definition of $\bclosure{A}$.
    Considering the Coxeter group generated by the $t_i$, since $r$ is a reflection, we can rewrite $r$ as the reduced expression (over the alphabet $\left\{t_i\right\}_{i \in [\ell]}$)
    \[
        r = (a, c_1) (c_1, c_2) \dots (c_{m-2}, c_{m-1})(c_{m-1}, b)(c_{m-2}, c_{m-1})\dots (c_1, c_2)(a, c_1).
    \]
    Recalling that the support of two reduced expressions of the same element are identical, we have that each $(c_{i-1}, c_i) \in A$.
    Although this word is a reduced expression over the alphabet $\left\{t_i\right\}_{i \in [\ell]}$, this does not guarantee that the length is always increasing after each right hand multiplication over $W$ itself.
    Therefore, it remains to show that $a < c_1 < \ldots < c_{m-1} < b$ implies that each (right) multiplication increases the length, \ie that
    \[
       e \xrightarrow{(a, c_1)} (a, c_1)
       \xrightarrow{(c_1, c_2)} (a, c_1)(c_1, c_2)
       \xrightarrow{(c_2, c_3)}
       \ldots
       \xrightarrow{(a, c_1)} r
    \] 
    is an $A$-path.
    We break this down into two parts.

    \textbf{First part:}
    \begin{align*}
        \perm
            {1 & \ldots & a & \ldots & c_1 &\ldots &  c_2 & \ldots & c_{i-1} &\ldots &  c_i &\ldots &  b &\ldots & n}
            {1 & \ldots & a & \ldots & c_1 &\ldots &  c_2 & \ldots & c_{i-1} &\ldots &  c_i &\ldots &  b &\ldots & n}
        &\xrightarrow{(a, c_1)}
        \perm
            {1 & \ldots & a & \ldots & c_1 &\ldots &  c_2 & \ldots & c_{i-1} &\ldots &  c_i &\ldots &  b &\ldots & n}
            {1 & \ldots & \re{c_1} & \ldots & \re{a} &\ldots &  c_2 & \ldots & c_{i-1} &\ldots &  c_i &\ldots &  b &\ldots & n}
        \\&\xrightarrow{(c_1, c_2)}
        \perm
            {1 & \ldots & a & \ldots & c_1 &\ldots &  c_2 & \ldots & c_{i-1} &\ldots &  c_i &\ldots &  b &\ldots & n}
            {1 & \ldots & \re{c_2} & \ldots & a &\ldots &  \re{c_1} & \ldots & c_{i-1} &\ldots &  c_i &\ldots &  b &\ldots & n}
        \\&\qquad\vdots
        \\&\xrightarrow{(c_{i-2}, c_{i-1})}
        \perm
            {1 & \ldots & a & \ldots & c_1 &\ldots &  c_2 & \ldots & c_{i-1} &\ldots &  c_i &\ldots &  b &\ldots & n}
            {1 & \ldots & \re{c_{i-1}} & \ldots & a &\ldots &  c_1 & \ldots & \re{c_{i-2}} &\ldots &  c_i &\ldots &  b &\ldots & n}
        \\&\xrightarrow{(c_{i-1}, c_{i})}
        \perm
            {1 & \ldots & a & \ldots & c_1 &\ldots &  c_2 & \ldots & c_{i-1} &\ldots &  c_i &\ldots &  b &\ldots & n}
            {1 & \ldots & \re{c_{i}} & \ldots & a &\ldots &  c_1 & \ldots & c_{i-2} &\ldots &  \re{c_{i-1}} &\ldots &  b &\ldots & n}
        \\&\qquad\vdots
        \\&\xrightarrow{(c_{m-1}, b)}
        \perm
            {1 & \ldots & a & \ldots & c_1 & \ldots &  c_2 & \ldots &  c_i     & \ldots & c_{m-1} & \ldots & b &\ldots & n}
            {1 & \ldots & \re{b} & \ldots & a   & \ldots &  c_1 & \ldots &  c_{i-1} & \ldots & c_{m-2} & \ldots & \re{c_{m-1}} &\ldots & n}.
    \end{align*}
    In the first step, we have $(c_1 - a) + (c_1 - a) = 2c_1 - 2a - 1$ inversions: $c_1 - a$ coming from $c_1$ being larger than every number in the interval $[a, c_1)$ and $c_1 - a$ coming from $a$ being smaller than every number in the interval $(a, c_1]$ and the $-1$ coming from counting $a < c_1$ twice.
    In each subsequent step we have $(c_i - a) + (c_i - c_{i-1}) + (c_{i-1} + c_{i-2}) + \ldots + (c_1 - a) - i$ inversions: $c_i$ being bigger than every number in $[a, c_i)$, $c_{i-1}$ being smaller than every number in the interval $(c_{i-1}, c_i]$, $c_{i-2}$ being smaller than every number in the interval $(c_{i-2}, c_{i-1}]$ etc. with each $c_k < c_i$ being counted twice, giving
    \[
        (c_i - a ) + (c_i - c_{i-1}) + (c_{i-1} + c_{i-2} ) + \ldots + (c_1 - a) - i = 2c_i - 2a - i > 2c_{i-1} -2a - i + 1.
    \] 
    In other words, our length is always increasing each time we multiply on the right.
    Therefore, at the end we have $2b - 2a - m$ inversions.

    \textbf{Second part:}
        \begin{align*}
        \perm
            {1 & \ldots & a & \ldots &  c_i     & \ldots & c_{i+1} & \ldots & c_{m-2} & \ldots & c_{m-1} & \ldots & b &\ldots & n}
            {1 & \ldots & b & \ldots &  c_{i-1} & \ldots & c_{i} & \ldots & c_{m-3} & \ldots & c_{m-2} & \ldots & c_{m-1} &\ldots & n}
        &\xrightarrow{(c_{m-1}, c_{m-2})}
        \perm
            {1 & \ldots & a & \ldots &  c_i     & \ldots & c_{i+1} & \ldots & c_{m-2} & \ldots & c_{m-1} & \ldots & b &\ldots & n}
            {1 & \ldots & b & \ldots &  c_{i-1} & \ldots & c_{i} & \ldots & c_{m-3} & \ldots & \re{c_{m-1}} & \ldots & \re{c_{m-2}} &\ldots & n}
        \\&\xrightarrow{(c_{m-2}, c_{m-3})}
        \perm
            {1 & \ldots & a & \ldots &  c_i     & \ldots & c_{i+1} & \ldots & c_{m-2} & \ldots & c_{m-1} & \ldots & b &\ldots & n}
            {1 & \ldots & b & \ldots &  c_{i-1} & \ldots & c_{i} & \ldots & \re{c_{m-2}} & \ldots & c_{m-1} & \ldots & \re{c_{m-3}} &\ldots & n}
        \\&\qquad\vdots
        \\&\xrightarrow{(c_{i}, c_{i+1})}
        \perm
            {1 & \ldots & a & \ldots &  c_i     & \ldots & c_{i+1} & \ldots & c_{m-2} &\ldots & c_{m-1} & \ldots & b &\ldots & n}
            {1 & \ldots & b & \ldots &  c_{i-1} & \ldots & \re{c_{i+1}} & \ldots & c_{m-2} &\ldots & c_{m-1} & \ldots & \re{c_{i}} &\ldots & n}
        \\&\xrightarrow{(c_{i-1}, c_{i})}
        \perm
            {1 & \ldots & a & \ldots &  c_i     & \ldots & c_{i+1} & \ldots & c_{m-2} & \ldots & c_{m-1} & \ldots & b &\ldots & n}
            {1 & \ldots & b & \ldots &  \re{c_{i}} & \ldots & c_{i+1} & \ldots & c_{m-2} & \ldots & c_{m-1} & \ldots & \re{c_{i-1}} &\ldots & n}
        \\&\qquad\vdots
        \\&\xrightarrow{(a, c_{1})}
        \perm
            {1 & \ldots & a & \ldots & c_1 & \ldots &  c_i     & \ldots & c_{i+1} & \ldots & c_{m-2} & \ldots & c_{m-1} & \ldots & b &\ldots & n}
            {1 & \ldots & b & \ldots &  \re{c_1} & \ldots & c_i & \ldots & c_{i+1} & \ldots & c_{m-2} & \ldots & c_{m-1} & \ldots & \re{a} &\ldots & n}
    \end{align*}

    In the first step, we add $(b - {c_{m-1}}) - (b - c_{m-1}) + 1 = 1$ inversions: adding $b - c_{m-1}$ inversions since $c_{m-2}$ is smaller than every number in the interval $[c_{m-1}, b)$, removing $b - {c_{m-1}}$ inversions since $c_{m-1}$ is larger than every number in the interval $(b, c_{m-1})$ and $+1$ since we add the inversion $(b, c_{m-1})$ and thus the length goes up.
    Notice that in each subsequent step, by a similar argument, we add $(c_{i} - {c_{i-1}}) - (c_{i} - c_{i-1}) + 1 = 1$ inversions.

    Thus, our length is always increasing, and we're left with length $2b - 2a - m + (m -1) = 2b - 2a - 1$ as we would expect.
    In other words, this reduced expression (in the alphabet $\{t_i\}$) gives an $A$-path for $r$.
\end{proof}

We are now in a place to prove our main theorem.
\begin{theorem}
    \label{thm:type_A_closure}
    Let $W$ be a type $A_n$ Coxeter group and let $A \subseteq T$.
    Then $\bclosure{A} = \bclosure{\bclosure{A}}$ and thus the Bruhat preclosure is a closure in type $A_n$.
\end{theorem}
\begin{proof}
    Let $r = (a, b)$ be an element in $\bclosure{\bclosure{A}} \bs \bclosure{A}$.
    Then by \autoref{lem:A-reflection-only-between}
    \begin{align*}
        r &= (a, c_1) (c_1, c_2) \dots (c_{m-2}, c_{m-1})(c_{m-1}, b)(c_{m-2}, c_{m-1})\dots (c_1, c_2)(a, c_1) \\
            &= t_1t_2 \ldots t_{m-1}t_{m}t_{m-1} \ldots t_2t_1
    \end{align*}
    where $a < c_1 < \ldots < c_{m-1} < b$, where length increases each time we multiply on the right by a reflection (as $r \in \bclosure{\bclosure{A}}$) and where the $t_i$ are contained in $\bclosure{A}$.
    As with $r$, if $t_i = (c_{i-1}, c_i) \in \bclosure{A} \bs A$ then $t_i$ must be equal to 
    \begin{align*}
        t_i &= (c_{i-1}, z_1)(z_1, z_2) \dots (z_{k-2}, z_{k-1})(z_{k-1}, c_i)(z_{k-2}, z_{k-1}) \dots (z_1, z_2)(c_{i-1}, z_1)\\
            &= t_1't_2'\ldots t_{k-1}'t_{k}'t_{k-1}'\ldots t_2't_1'
    \end{align*}
    by \autoref{lem:A-reflection-only-between} with each $t_j' \in A$ and $c_{i-1} < z_1 < \dots < z_{k-1} < c_i$.
    Notice that we can't substitute $t_1't_2'\ldots t_{k-1}'t_{k}'t_{k-1}'\ldots t_2't_1'$ for $t_i$ directly into the equation of $r$ as then sometimes the length will decrease.
    But instead, we can replace $t_i$ with either the first half or the second half accordingly
    \begin{gather*}
        t_{i-1}t_it_{i+1} = (c_{i-2}, c_{i-1})(c_{i-1}, c_i)(c_i, c_{i+1}) \\
            \to  \\
        (c_{i-2}, c_{i-1})(c_{i-1}, z_1)(z_1, z_2) \dots (z_{k-2}, z_{k-1})(z_{k-1}, c_i)(c_i, c_{i+1}) = t_{i-1}t_1't_2'\ldots t_{k}'t_{i+1}
    \end{gather*}
    and
    \begin{gather*}
        t_{i+1}t_it_{i-1} = (c_i, c_{i+1})(c_{i-1}, c_i)(c_{i-2}, c_{i-1}) \\
        \to  \\
        (c_i, c_{i+1})(z_{k-1}, c_i)(z_{k-2}, z_{k-1}) \dots (z_1, z_2)(c_{i-1}, z_1)(c_{i-2}, c_{i-1}) = t_{i+1}t_{k+1}'\ldots t_2't_1't_{i-1}.
    \end{gather*}
    Replacing each $t_i \in \bclosure{A} \bs A$ in this way and using a similar argument as in the previous proof, we have an $A$-path for $r$ and therefore $r \in \bclosure{A}$.
    Thus, $\bclosure{\bclosure{A}} = \bclosure{A}$ as desired.
\end{proof}

Thus, combining \autoref{thm:type_A_closure} and  \autoref{thm:main} we obtain that Dyer's conjecture is true in type $A$.
\begin{corollary}
    \label{cor:type_A_Dyer}
    Let $W$ be a type $A_n$ Coxeter group and let $u$ and $v$ be elements in $W$.
    Then
    \[
        N(u \joinR v) = \bclosure{N(u) \cup N(v)}.
    \] 
\end{corollary}

\bibliographystyle{abbrv}
\bibliography{paper}
\label{sec:biblio}

\end{document}

%% file: figures/H3.tex
\begin{tikzpicture}
    [vertex/.style={draw=black, anchor=base, rectangle, fill=white, minimum width=1cm, minimum height=0.5cm},
        vertexA/.style={draw=blue, anchor=base, rectangle, fill=white, minimum width=1cm, minimum height=0.5cm},
        vertexRefA/.style={draw=black, anchor=base, rectangle, fill=black!15!white, minimum width=1cm, minimum height=0.5cm, ultra thick},
        vertexRefB/.style={draw=blue!95!black, anchor=base, rectangle, fill=blue!15!white, minimum width=1cm, minimum height=0.5cm, ultra thick},
        vertexRefC/.style={draw=red!95!black, anchor=base, rectangle, fill=red!15!white, minimum width=1cm, minimum height=0.5cm, ultra thick},
        newA/.style={blue, ultra thick, dashed},
        yscale=2,
    ]

    \coordinate (e) at (0, 0);
    \coordinate (r) at (-2, 1);
    \coordinate (srs) at (3.5, 1);
    \coordinate (rsrs) at (-3, 2);
    \coordinate (srsr) at (4.5, 2);
    \coordinate (rsrsr) at (0, 3);
    \coordinate (tsrstr) at (-6, 3);
    \coordinate (rtsrstr) at (3.5, 3);
    \coordinate (stsrstrsr) at (-6,4);
    \coordinate (tsrstrsrs) at (-3,4);
    \coordinate (strstrstr) at (7,4.25);
    \coordinate (rstsrstrsr) at (3.5,4.75);
    \coordinate (rtsrstrsrs) at (2.25,4.5);
    \coordinate (strstrsrst) at (4.75,4.5);
    \coordinate (stsrstrsrs) at (-6,4.5);
    \coordinate (rstrstrstr) at (0.25,4.5);
    
    \coordinate (strsrstrsr) at (7,5);
    \coordinate (rstrsrstrsr) at (1.5, 6);
    \coordinate (rstrstrsrst) at (-1,5.75);
    \coordinate (rstsrstrsrs) at (-3.5,6);
    \coordinate (strstrstrsr) at (7,6);
    \coordinate (rstrstrstrsr) at (0,7);
    \coordinate (strstrstrsrs) at (-3.5,7);
    \coordinate (strstrstrstr) at (5,7);
    \coordinate (rstrstrstrsrs) at (-6,6);
    \coordinate (rstrstrstrstr) at (7,8);
    \coordinate (strstrstrstrs) at (3, 7.25);
    \coordinate (rstrstrstrstrs) at (-2,8);

    \draw (e) -- (r); 
    \draw (e) -- (srs); 
    \draw[newA] (e) -- (rsrsr); 
    \draw (e) -- (rtsrstr); 
    \draw (e) -- (strstrstrstrs); 

    \draw (r) -- (rsrs); 
    \draw (r) -- (rstrstrstrstrs); 
    \draw (r) -- (tsrstr); 
    \draw[newA] (r) -- (srsr); 

    \draw (srs) -- (srsr); 
    \draw (srs) -- (rtsrstrsrs); 
    \draw (srs) -- (strstrsrst); 

    \draw (rsrs) -- (rsrsr); 
    \draw (rsrs) -- (rstrstrsrst); 
    \draw (rsrs) -- (tsrstrsrs); 

    \draw (srsr) -- (rstsrstrsrs); 
    \draw (srsr) -- (strstrstr); 

    \draw (rsrsr) -- (stsrstrsrs); 
    \draw (rsrsr) -- (rstrstrstr); 

    \draw (rtsrstr) -- (strstrsrst); 
    \draw (rtsrstr) -- (rtsrstrsrs); 
    \draw[newA] (rtsrstr) -- (rstsrstrsr); 

    \draw (tsrstr) -- (rstrstrsrst); 
    \draw (tsrstr) -- (tsrstrsrs); 
    \draw[newA] (tsrstr) -- (stsrstrsr); 

    \draw (rstrstrstr) -- (rstrstrstrsrs); 
    \draw[newA] (rstrstrstr) -- (rstrsrstrsr); 

    \draw (rstrsrstrsr) -- (rstrstrstrsr); 
    \draw (stsrstrsrs) -- (rstrstrstrsrs); 
    \draw (strsrstrsr) -- (strstrstrsr); 
    \draw (rstsrstrsrs) -- (strstrstrsrs); 
    \draw (strstrstrstr) -- (rstrstrstrstr); 
    \draw (strstrstrstrs) -- (rstrstrstrstrs); 

    \draw (strstrstr) -- (strstrstrsrs); 
    \draw[newA] (strstrstr) -- (strsrstrsr); 

    \draw (rstrstrsrst) -- (rstrstrstrstrs); 
    \draw[newA] (rstrstrsrst) -- (strstrstrstr); 

    \draw (tsrstrsrs) -- (rstrstrstrstrs); 
    \draw (tsrstrsrs) -- (stsrstrsrs); 

    \draw (strstrsrst) -- (strstrstrstrs); 
    \draw[newA] (strstrsrst) -- (rstrstrstrstr); 

    \draw (rtsrstrsrs) -- (strstrstrstrs); 
    \draw (rtsrstrsrs) -- (rstsrstrsrs); 

    \node[vertex] at (e) {$e$};
    \node[vertexRefA] at (r) {$r$};
    \node[vertexRefA] at (srs) {$srs$};
    \node[vertex] at (rsrs) {$rsrs$};
    \node[vertex] at (srsr) {$srsr$};
    \node[vertexRefB] at (rsrsr) {$rsrsr$};
    \node[vertexRefA] at (rtsrstr) {$rtsrstr$};
    \node[vertexRefA] at (strstrstrstrs) {$strstrstrstrs$};
    \node[vertexRefC] at (rstrsrstrsr) {$rstrsrstrsr$};

    \node[vertex] at (tsrstr) {$tsrstr$};
    \node[vertexA] at (stsrstrsr) {$stsrstrsr$};
    \node[vertex] at (tsrstrsrs) {$tsrstrsrs$};
    \node[vertex] at (strstrstr) {$strstrstr$};
    \node[vertex] at (rstsrstrsr) {$rstsrstrsr$};
    \node[vertex] at (rtsrstrsrs) {$rtsrstrsrs$};
    \node[vertex] at (strstrsrst) {$strstrsrst$};
    \node[vertex] at (stsrstrsrs) {$stsrstrsrs$};
    \node[vertex] at (rstrstrstr) {$rstrstrstr$};
    \node[vertexA] at (rstsrstrsr) {$rstsrstrsr$};
    \node[vertexA] at (strsrstrsr) {$strsrstrsr$};
    \node[vertex] at (rstrstrsrst) {$rstrstrsrst$};
    \node[vertex] at (rstsrstrsrs) {$rstsrstrsrs$};
    \node[vertexA] at (strstrstrsr) {$strstrstrsr$};
    \node[vertexA] at (rstrstrstrsr) {$rstrstrstrsr$};
    \node[vertex] at (strstrstrsrs) {$strstrstrsrs$};
    \node[vertexA] at (strstrstrstr) {$strstrstrstr$};
    \node[vertex] at (rstrstrstrsrs) {$rstrstrstrsrs$};
    \node[vertexA] at (rstrstrstrstr) {$rstrstrstrstr$};
    \node[vertex] at (rstrstrstrstrs) {$rstrstrstrstrs$};
\end{tikzpicture}

%% file: paper.bib
@Book{Bjorner_CombinatoricsofCoxeterGroups,
  Title                    = {Combinatorics of Coxeter Groups},
  Author                   = {Bjorner, A. and Brenti, F.},
  Publisher                = {Springer Berlin Heidelberg},
  Year                     = {2005},
  Series                   = {Graduate Texts in Mathematics}
}

@Book{Humphreys_ReflectionGroupsandCoxeterGroups,
  Title                    = {Reflection Groups and Coxeter Groups},
  Author                   = {Humphreys, J.E.},
  Publisher                = {Cambridge University Press},
  Year                     = {1992},
  Series                   = {Cambridge Studies in Advanced Mathematics},

  ISBN                     = {9780521436137},
  Lccn                     = {93109267},
  Url                      = {https://books.google.ca/books?id=ODfjmOeNLMUC}
}

@article{Dyer_OnTheWeakOrderOfCoxeterGroups,
  author    = {Dyer, Matthew},
  journal   = {Canadian Journal of Mathematics},
  title     = {On the weak order of {C}oxeter groups},
  year      = {2019},
  issn      = {1496-4279},
  month     = jan,
  number    = {2},
  pages     = {299--336},
  volume    = {71},
  day       = {29},
  doi       = {https://doi.org/10.4153/CJM-2017-059-0},
  publisher = {Canadian Mathematical Society},
}

@article{Dyer_HeckeAlgebrasAndShellingsOfBruhatIntervals,
    Title                   = {Hecke algebras and shellings of {B}ruhat intervals},
    Author                  = {Matthew Dyer},
    Journal                 = {Compositio Mathematica},
    Year                    = {1993},
    Number                  = {1},
    Pages                   = {91 - 115},
    Volume                  = {89}
}

@article{Dyer_HeckeAlgebrasAndShellingsOfBruhatIntervalsIITWistedBruhatOrders,
    Title                   = {Hecke algebras and shellings of {B}ruhat intervals II; twisted {B}ruhat orders},
    Author                  = {Matthew Dyer},
    Journal                 = {Contemporary Mathematics},
    Volume                  = {139},
    Year                    = {1992},
    Pages                   = {141-165}
}

@article{Dyer_ReflectionSubgroupsOfCoxeterSystems,
    Title                   = {Reflection subgroups of Coxeter systems},
    Author                  = {Matthew Dyer},
    Journal                 = {J. of Alg.},
    Volume                  = {135},
    Year                    = {1990},
    Pages                   = {57-73}
}

@misc{Dyer_Communication,
    author="Dyer, Matthew",
    date="2025-11-30",
    howpublished="Private communication",
}

@Book{Birkhoff_LatticeTheory,
    Title                   = {Lattice theory},
    Author                  = {G. Birkhoff},
    Publisher               = {American Mathematical Society},
    Year                    = {1979},
    Edition                 = {3}
}

@article{Edgar_DominanceAndRegularityInCoxeterGroups,
    Title                   = {Dominance and Regularity in {C}oxeter groups},
    Author                  = {Tom Edgar},
    Year                    = {2009},
    Month                   = jul
}

@inproceedings{Biagioli2025,
  author        = {Biagioli, Riccardo and Perrone, Lorenzo},
  booktitle       = {Séminaire Lotharingien de Combinatoire, 93B.13 (2025)},
  title         = {On a Conjecture of Dyer on the Join in the Weak Order of a Coxeter group},
  year          = {2025},
  month         = oct,
  archiveprefix = {arXiv},
  copyright     = {arXiv.org perpetual, non-exclusive license},
  doi           = {10.48550/ARXIV.2510.11446},
  eprint        = {2510.11446},
  primaryclass  = {math.CO},
  publisher     = {arXiv},
}

@misc{Biagioli_Communication,
    author="Biagioli, Riccardo",
    date="2025-11-29",
    howpublished="private communication",
}
